\documentclass[sn-mathphys-num]{sn-jnl}

\usepackage{graphicx}%
\usepackage{multirow}%
\usepackage{amsmath,amssymb,amsfonts}%
\usepackage{amsthm}%
\usepackage{mathrsfs}%
\usepackage[title]{appendix}%
\usepackage{xcolor}%
\usepackage{textcomp}%
\usepackage{manyfoot}%
\usepackage{booktabs}%
\usepackage{algorithm}%
\usepackage{algorithmicx}%
\usepackage{algpseudocode}%
\usepackage{listings}%
\usepackage{booktabs}
\newtheorem{lemma}{Lemma}

\theoremstyle{thmstyleone}%
\newtheorem{theorem}{Theorem}
\theoremstyle{thmstyletwo}%
\newtheorem{remark}{Remark}%

\theoremstyle{thmstylethree}%

\begin{document}

\title[Article Title]{Weak Convergence Analysis for the Finite Element  Approximation to Stochastic Allen-Cahn Equation Driven by Multiplicative White Noise}


\author[1]{\fnm{Minxing} \sur{Zhang}}\email{Zhangminxingjlu@163.com}

\author[1]{\fnm{Yongkui} \sur{Zou}}\email{zouyk@jlu.edu.cn}

\author[1]{\fnm{Ran} \sur{Zhang}}\email{zhangran@jlu.edu.cn}

\author*[2]{\fnm{Yanzhao} \sur{Cao}}\email{yzc0009@auburn.edu}

\affil[1]{\orgdiv{School of Mathematics}, \orgname{Jilin University}, \orgaddress{\city{Changchun}, \postcode{130012}, \country{P.R. China}}}

\affil*[2]{\orgdiv{Department of Mathematics and Statistics}, \orgname{Auburn University}, \orgaddress{\city{Auburn}, \postcode{AL 36849}, \country{USA}}}


\abstract{In this paper, we aim to study the optimal weak convergence order for the finite element approximation to a stochastic Allen-Cahn equation driven by multiplicative white noise. We first construct an auxiliary equation based on the splitting-up technique and derive prior estimates for the corresponding Kolmogorov equation and obtain the strong convergence order of 1 in time between the auxiliary and exact solutions. Then, we prove the optimal weak convergence order of the finite element approximation to the stochastic Allen-Cahn equation by deriving the weak convergence order between the finite element approximation and the auxiliary solution via the theory of Kolmogorov equation and Malliavin calculus. Finally, we present a numerical experiment to illustrate the theoretical analysis.}





\maketitle

\section{Introduction}
\label{Sec1}
Over the past decades, numerical analysis of stochastic partial differential equations (SPDEs) has attracted increasing attention \cite{CaiMeng2024,CaoYanzhao2020,ChaiShimin2018,Jentzen2009,Lord2013,ZhangFengshan2022,ZhangFengshan2024}. Both strong and weak convergence rates of the numerical approximation for SPDEs under globally Lipschitz continuity condition have been widely studied  \cite{Andersson2016,Debussche2011,Kovacs2012,Kruse2014,YanYubin2005}. However, many practical models, such as the stochastic Allen-Cahn equation to be discussed in this work,   fail to satisfy such a condition, which motivates studies on SPDEs with non-globally Lipschitz coefficient \cite{Kovacs2018_2,Jentzen2020,Jentzen2020_2,XuChuanju2023}. 

The stochastic Allen-Cahn equation models the effect of thermal perturbations, and plays an important role in phase-field theory and the simulations of rare events in infinite dimensional stochastic systems \cite{Cerrai2003}.
Numerical approximations for the stochastic Allen-Cahn equations have been studied extensively in the 
last decade. Initial studies focused on the strong convergence of numerical approximation
to the stochastic Allen-Cahn equation driven by additive noise. Kov\'acs et al.\ \cite{Larsson2015,Kovacs2018} studied  the strong convergence order for a temporal semi-discretization with an implicit Euler method. Br\'{e}hier et al.\ \cite{Brehier2019_2,Brehier2019} introduced an explicit time discretization scheme based on a splitting strategy and analyzed its strong convergence order. Qi and Wang \cite{QiRuisheng2019} analyzed the optimal strong convergence order of a full discretization using the FEM and the implicit Euler method. Wang \cite{WangXiaojie2020} applied a spectral Galerkin method and a tamed accelerated exponential Euler method to set up a full discretization and proved the strong convergence order. Recently, the stochastic Allen-Cahn equation driven by multiplicative noise has also been studied. Liu and Qiao \cite{LiuZhihui2021} constructed a drift-implicit Euler-Galerkin scheme and derived the optimal strong convergence order. Huang and Shen \cite{Shenjie2023} applied the  spectral Galerkin method and tamed semi-implicit Euler method to set up a fully discretized approximation, and analyze its strong convergence order as well as unconditional stability. Yang et al.\ \cite{ZhaoWeidong2024} discretized stochastic Allen-Cahn equation by the implicit Euler scheme in time and the discontinuous Galerkin method in space and obtained its optimal strong convergence rate. 

In this paper, we are concerned with the weak convergence of numerical approximations of the stochastic Allen Cahn equation driven by multiplicative noise. The weak error, sometimes more relevant in various fields such as finance and engineering, concerns with the approximation of the probability distribution of solutions to SPDEs \cite{Kloeden1992,Milstein2004}. It measures the error made by
sampling from an approximate probability law of the exact solution, rather than the deviation
from trajectory of the exact solution, as for the strong error \cite{Andersson2016,Hairer2018}. 
In recent years,  several stratgies have been proposed to analyze weak errors of numerical approximation to SPDEs with Lipschitz continuous coefficients. These stratgies include Kolmogorov equation method \cite{Andersson2016,Brehier2018,Naurois2021,Wangxiaojie2013}, the mild It\^o formula \cite{Jentzen2019,Jentzen2021} and others \cite{Andersson2016_2,WangXiaojie2016}. Cui et al.\  constructed a full discretization scheme by FEM and backward Euler method, and obtained the optimal weak convergence order for spatial semi-discretization \cite{CuiJianbo2019} and for full discretization \cite{CuiJianbo2021}. Cai et al.\ \cite{WangXiaojie2021} discretized the stochastic Allen-Cahn equation driven by an additive noise using a spectral Galerkin method and a tamed exponential Euler method and then investigated the weak convergence order. 

To the best of our knowledge, there has been no prior work on weak error estimates for the finite element method (FEM) applied to the stochastic Allen-Cahn equation driven by multiplicative noise. Addressing this gap is the primary focus of this work. The main analytical challenge lies in deriving high-order, uniform a priori estimates for the Fréchet derivatives of the solution to the associated Kolmogorov equation. In case where the nonlinear forcing  term of a SPDE satisfies a uniform Lipschitz condition \cite{Andersson2016}, the Fréchet derivatives of Kolmogorov equation with respect to initial value under the action of a fractional power of the Laplacian operator satisfies an SPDE with uniform bounded coefficient functions, which can be easily dealt with. For the Allen-Cahn equation driven by an additive noise \cite{CuiJianbo2019}, the counterpart satisfies a  PDE without a noise forcing term. In this case, The techniques for deterministic PDE can be applied to obtain uniform estimates for the Fréchet derivatives  under the action of a fractional power of the Laplacian operator, and  these estimates directly yield the weak convergence order.

The Allen-Cahn equation driven by multiplicative noise introduces additional challenges in analyzing prior estimates for the solution of the Kolmogorov equation. Specifically, the Fréchet derivatives of the Kolmogorov equation's solution with respect to the initial value satisfy an SPDE driven by multiplicative noise, characterized by a one-sided Lipschitz coefficient function.  A key innovation of our work is the application of the mild Itô formula \cite{DaPrato2019} to the Kolmogorov equation. This approach enables us to establish the uniform boundedness of the first- and second-order Fréchet derivatives. Furthermore, to achieve the desired weak convergence order, we introduce a novel method that transfers the action of the fractional power of the Laplacian operator from the Kolmogorov equation's solution to the mild solution of the FEM approximation. This technique allows us to rigorously prove the weak convergence order for the stochastic FEM approximation.

Our strategy for obtaining the optimal weak convergence order consists of two main steps:  
\begin{enumerate}
    \item \textbf{Temporal Splitting and Auxiliary SPDE:} We construct a temporal splitting approximation for the stochastic Allen-Cahn equation and define an auxiliary SPDE. We establish a strong convergence order of 1 in time between the exact solution and the auxiliary solution and derive the regularity estimates for the corresponding Kolmogorov equation.  
    \item \textbf{Weak Convergence Analysis:} Using Malliavin calculus and the Kolmogorov equation, we analyze the weak convergence order between the FEM approximation and the auxiliary solution. The optimal weak convergence order between the FEM approximation and the exact solution follows as a direct result.  
\end{enumerate}

We note that while we achieve a first-order convergence rate in the weak sense, only a half-order convergence rate is attainable in the strong sense (\cite{WangXiaojie2020, WangXiaojie2021}).

The rest of this paper is organized as follows. In section {\ref{Sect2}}, we introduce some preliminaries and main assumptions, and give a brief introduction to Malliavin calculus. In section {\ref{Sect3}}, we construct an auxiliary equation based on splitting-up technique, and investigate the strong convergence order in time between the exact and auxiliary solutions. In section {\ref{Sect4}}, we construct an FEM semi-discretization, and present a regularity estimate of its Malliavin derivative. In section {\ref{Sect5}}, we analyze the weak convergence order of the finite element semi-discretization to the auxiliary equation. In section {\ref{Sect6}}, we carry out a numerical experiment to demonstrate the theoretical analysis.

\section{Preliminaries}
\label{Sect2}

\subsection{Some spaces and properties}
Let $(U,(\cdot, \cdot)_{U})$ and $(V,(\cdot, \cdot)_{V})$ be two real separable Hilbert spaces and $(\Omega,\mathcal{F},\mathbb{P})$ be a complete probability space. By $\mathcal{L}(U,V)$ we denote the Banach space of all bounded linear operators from $U$ to $V$ endowed with operator norm. For simplicity, we write $\mathcal{L}(U) = \mathcal{L}(U,U)$. If $T \in \mathcal{L}(U)$, then $\|T\|_{\mathcal{L}(U)} = \|T^*\|_{\mathcal{L}(U)}$.

For any $T \in \mathcal{L}(U)$, let $\{\sigma_k\}_{k \in \mathbb{N}}$ be eigenvalues of $(TT^*)^{\frac{1}{2}}$. For any $1\le p<\infty$, define a Banach space of linear operators
\begin{equation*}
	\mathcal{L}_p(U) = \big\{T\in\mathcal{L}(U):\ \|T\|_{\mathcal{L}_p(U)} = \big(\sum_{k=1}^{\infty}\sigma_k^p\big)^{\frac{1}{p}}<\infty\big\}.
\end{equation*}
Let $\{e_k\}_{k=1}^{\infty}$ be an orthonormal basis of $U$. Define $\mathrm{Tr}(T) = \sum_{k=1}^{\infty}(Te_{k}, e_{k})_{U}$ if the summation is finite. Obviously, $\mathcal{L}_1(U)$ is the space of trace class operators. For any $T \in \mathcal{L}_1(U)$, we have $|\mathrm{Tr}(T)| \le \|T\|_{\mathcal{L}_1(U)}$ and the equality holds if $T$ is positive semidefinite. Furthermore, for any $S\in \mathcal{L}(U)$, there hold \cite{DaPrato2014,Debussche2011}
\begin{equation}\label{eq2.2.0}
	\begin{aligned}
		&\mathrm{Tr}(T)=\mathrm{Tr}(T^{*}), \ \mathrm{Tr}(TS)=\mathrm{Tr}(ST),\\
		&\mathrm{Tr}(ST)\le\|ST\|_{\mathcal{L}_1(U)}\le\|S\|_{\mathcal{L}(U)}\|T\|_{\mathcal{L}_1(U)},\\ &\mathrm{Tr}(TS)\le\|TS\|_{\mathcal{L}_1(U)}\le\|S\|_{\mathcal{L}(U)}\|T\|_{\mathcal{L}_1(U)}.
	\end{aligned}
\end{equation}

By $\mathcal{L}_2(U,V)$ we denote the space of Hilbert-Schmidt operators from $U$ to $V$, which is also a Hilbert space with inner product and norm
\begin{equation}\label{eq2.2.0.1}
	( S, T)_{\mathcal{L}_{2}(U,V)} =\sum_{k=1}^{\infty}( S e_{k}, T e_{k})_{V}=\mathrm{Tr}(T^{*} S),\ 
	\|T\|^2_{\mathcal{L}_{2}(U,V)} = \mathrm{Tr}(TT^*).
\end{equation}
For simplifying notations, we denote by $\mathcal{L}_2(U) = \mathcal{L}_2(U,U)$. Then for any $S\in \mathcal{L}(U)$, $T_1,T_2 \in \mathcal{L}_2(U)$, there hold that \cite{DaPrato2014,Debussche2011}
\begin{equation}\label{eq2.2.1}
	\begin{aligned}
		&\|T_1\|_{\mathcal{L}(U)} \le \|T_1\|_{\mathcal{L}_2(U)},\\
		&\|ST_1\|_{\mathcal{L}_2(U)}\le\|S\|_{\mathcal{L}(U)}\|T_1\|_{\mathcal{L}_2(U)},\\ &\|T_1S\|_{\mathcal{L}_2(U)}\le\|S\|_{\mathcal{L}(U)}\|T_1\|_{\mathcal{L}_2(U)},\\ &\|T_1T_2\|_{\mathcal{L}_1(U)}\le\|T_1\|_{\mathcal{L}_2(U)}\|T_2\|_{\mathcal{L}_2(U)}.
	\end{aligned}
\end{equation}

For any integer $m\ge 1$, by $\mathcal{C}_b^m(U)$ we denote the space of not necessarily bounded functions $\phi: U \rightarrow \mathbb{R}$ having $m$-th continuous and bounded Fr\'echet derivatives $D\phi, D^{2}\phi, \cdots, D^{m}\phi$. We endow it with the seminorm $|\cdot|_{\mathcal{C}_b^m(U)}$, defined as the smallest constant $C \ge 0$ satisfying
\begin{equation*}
	\sup _{u \in U}|D^{m} \phi(u)(h_{1}, \cdots, h_{m})| \le C\|h_{1}\|_{U} \cdots\|h_{m}\|_{U}, \quad \forall h_{1}, \cdots, h_{m} \in U.
\end{equation*}
Let $D\phi(u)$ denote the first order Fr\'echet derivative of $\phi$ at $u\in U$, which is a bounded linear functional
on $U$. According to Riesz
representation theorem, there exists a unique element in $U$, still denoted by $D\phi(u)$, such that for any $h \in U$
\begin{equation}\label{eq2.2.1.1}
	D\phi(u)h = ( D\phi(u),h)_U,\quad |D\phi(u)h| \le |\phi|_{\mathcal{C}_b^1(U)}\|h\|_U.
\end{equation}

\subsection{Stochastic Allen-Cahn equation and assumptions}
In this section, we first describe the stochastic Allen-Cahn equation driven by multiplicative white noise and assumptions, and then collect some properties of its solution.

Let $L$ be a positive number and define $\mathcal{D} = (0,L)$. By $H = L^2(\mathcal{D})$ we denote the standard Sobolev space with inner product $(\cdot,\cdot)$ and norm $\|\cdot\|$. Let $V = \mathcal{C}(\bar{\mathcal{D}})$ be the space of continuous functions endowed with maximum norm $\|\cdot\|_V$. 

Consider a semilinear parabolic SPDE with homogeneous Dirichlet boundary condition driven by multiplicative white noise
\begin{equation}\label{eq2.1.1}
	\begin{aligned}
		dX(t,x)-\Delta X(t,x)dt =&\ F(X(t,x))dt+G(X(t,x))dW(t),\ t\in (0,T],x\in \mathcal{D},\\
		X(t,x)=&\ 0,\ t\in (0,T],x\in \partial \mathcal{D},\\
		\quad X(0,x) =&\ \xi(x),\ x\in \mathcal{D},
	\end{aligned}
\end{equation}
where $T>0$ is fixed.
By $X(t,x,\xi)$ we denote the mild solution of \eqref{eq2.1.1} satisfying $X(0,x,\xi) = \xi(x)$. For simplifying notations, we abbreviate it to $X(t,\xi)$.

Denote by $A =-\Delta:\mathrm{Dom}(A)\subset H \to H$ the Laplacian operator under homogeneous Dirichlet boundary condition, where $\operatorname{Dom}(A) = H_0^1(\mathcal{D}) \cap H^2(\mathcal{D})$. Then, the operator $A$ possesses a family of eigenvalues $\lambda_i = \pi^2i^2$ with complete orthonormal eigenbases $\varphi_i(x)=\sqrt{2}\sin(i\pi x) \in H$. For any $\alpha \in \mathbb{R}$, we define the fractional powers of $A$ by
\begin{equation*}
	A^{\alpha}v = \sum_{i=1}^{\infty}\lambda_i^\alpha(v,\varphi_i)\varphi_i.
\end{equation*}
Let $\dot{H}^{\alpha} = \{v \in H: \| v \|_{\alpha} =\| A^{\frac{\alpha}{2}}v\| = \big( \sum_{i=1}^\infty \lambda_i^\alpha( v, \varphi_i)^2 \big)^\frac{1}{2}<\infty\}$. Then $(\dot{H}^{\alpha},\| \cdot \|_{\alpha})$ is a Hilbert space. Especially, we have $\dot{H}^0 = H$, $\dot{H}^1(\mathcal{D}) = H_0^1(\mathcal{D})$ and $\dot{H}^2(\mathcal{D}) = H_0^1(\mathcal{D}) \cap H^2(\mathcal{D})$, cf.\ \cite{Thomee2006}. For any $\beta>\frac{1}{4}$, direct computation leads to 
\begin{equation}\label{eq2.2.00}
	\|A^{-\beta}\|^2_{\mathcal{L}_2(H)} = \sum_{j=1}^\infty \|(\pi^2j^2)^{-\beta}\varphi_j\|^2 \le \sum_{j=1}^\infty j^{-4\beta}<\infty,
\end{equation} 
which implies $A^{-\beta}$ is a Hilbert-Schmidt operator.

The operator $A$ generates an analytic semigroup $S(t) = e^{tA}$ on $H$. According to \cite{Thomee2006}, for any $\alpha \ge 0$, there hold
\begin{equation}\label{eq2.2.2}
	A^{\alpha}S(t) = S(t)A^{\alpha} \ \text{ on } \dot{H}^{2\alpha}  \text{ and } 	\|A^{\alpha}S(t)\|_{\mathcal{L}(H)} \le Ct^{-\alpha}, \quad t>0.
\end{equation}
where $C=C(\alpha)>0$ is a constant.

Now we introduce assumptions used in this article.

\begin{itemize}
	\item[{\bf(A1)}] $F: L^{6}(\mathcal{D}) \rightarrow H$ is a Nemytskii operator defined by
	\begin{equation*}
		F(u)(x)=f(u(x))=u(x)-u^{3}(x),\  \forall x \in \mathcal{D}, u \in L^{6}(\mathcal{D})
	\end{equation*}
\end{itemize}
According to \cite{WangXiaojie2021}, there hold estimates
\begin{equation}\label{eq2.2.5}
	\begin{aligned}
		& \|F(u)\|_1 \le C_F(1+\|u\|_V^2)\|u\|_1,\quad \forall u \in \dot{H}^1,\\
		& \|F(u)-F(v)\| \leq C_F(1+\|u\|_{V}^{2}+\|v\|_{V}^{2})\|u-v\|, \quad \forall u, v \in V,\\
		& ( DF(u)v,v) \le C_F\|v\|^2,\quad \forall u,v \in L^{6}(\mathcal{D}), \\
		& \|DF(u)v\| \leq C_F(1+\|u\|_{V}^{2})\|v\|, \quad \forall u \in V, v \in L^{6}(\mathcal{D}), \\
		& \|D^2F(u)(v,w)\| \leq C_F\|u\|_V\|v\|\|w\|, \quad \forall u \in V, v,w \in L^{6}(\mathcal{D}), \\
	\end{aligned}
\end{equation}
where $C_F>0$ is a constant.

\begin{itemize}
	\item[{\bf(A2)}] $G(u) = B_0+B_1u+g(u)$, where $B_0 \in \mathcal{L}(H)$, $B_1 \in \mathcal{L}(\dot{H}^{-\frac{1}{2}-\delta},\mathcal{L}(H))$, $g(0) = 0$ and $g \in \mathcal{C}_b^2(\dot{H}^{-\frac{1}{2}-\delta},\mathcal{L}(H))$ for a small $\delta >0$.
\end{itemize}

\begin{remark}
	The requirement $g(0)=0$ is natural, otherwise we can simultaneously replace $B_0$ and $g(u)$ by $B_0+g(0)$ and $g(u)-g(0)$, respectively.
\end{remark}

The following lemma collects some properties about $G$.

\begin{lemma}\label{l2.2.1}
	For any $u\in \dot{H}^{-\frac{1}{2}-\delta}$ and $v,w \in H$, there exists a constant $C>0$, such that
	\begin{equation*}
		\begin{aligned}
			&\|G(u)\|_{\mathcal{L}(H)} \le C(1+\|u\|_{-\frac{1}{2}-\delta}),\\
			&\|DG(u)v\|_{\mathcal{L}_2(H)} \le C\|v\|,\\ 
			&\|D^2G(u)(v,w)\|_{\mathcal{L}_2(H)} \le C\|v\|\|w\|.
		\end{aligned}
	\end{equation*} 
	Moreover, for any $u \in \dot{H}^1$, there exists a constant $C>0$, such that
	\begin{equation*}
		\|A^{\frac{1}{2}}G(u)\|_{\mathcal{L}_2(H)}\le C\|u\|_1.
	\end{equation*}
\end{lemma}

\begin{proof}
	According to {\bf(A2)}, for any $u \in \dot{H}^{-\frac{1}{2}-\delta}$, there exists a constant $C_G>0$, such that $\|B_0\|_{\mathcal{L}(H)} \le C_G$, $\|B_1u\|_{\mathcal{L}(H)} \le C_G\|u\|_{-\frac{1}{2}-\delta}$, $\|Dg(u)\|_{\mathcal{L}(\dot{H}^{-\frac{1}{2}-\delta},\mathcal{L}(H))} \le C_G$ and
	\begin{equation*}
		\|g(u)\|_{\mathcal{L}(H)} = \|g(u)-g(0)\|_{\mathcal{L}(H)} = \|\int_0^1Dg(\theta u) ud\theta\|_{\mathcal{L}(H)} \le C_G\|u\|_{-\frac{1}{2}-\delta}
	\end{equation*}
	Thus we have
	\begin{equation*}
		\begin{aligned}
			&\ \|G(u)\|_{\mathcal{L}(H)} \le \|B_0\|_{\mathcal{L}(H)}+\|B_1u\|_{\mathcal{L}(H)}+\|g(u)\|_{\mathcal{L}(H)} \le C(1+\|u\|_{-\frac{1}{2}-\delta}).
		\end{aligned}
	\end{equation*}
	Furthermore, if $u \in H$ and noticing $H \hookrightarrow \dot{H}^{-\frac{1}{2}-\delta}$, there holds
	\begin{equation}\label{eq2.2.1.2}
		\|G(u)\|_{\mathcal{L}(H)} \le C_G(1+\|u\|_{-\frac{1}{2}-\delta}) \le C(1+\|u\|).
	\end{equation}
	
	
	Noticing that $A^{\frac{1}{4}+\frac{1}{2}\delta} \in \mathcal{L}(H,\dot{H}^{-\frac{1}{2}-\delta})$, then we have
	\begin{equation*}
		\begin{aligned}
			\|B_1A^{\frac{1}{4}+\frac{1}{2}\delta}\|_{\mathcal{L}(H,\mathcal{L}(H))} =&\ \sup_{u\in H}\frac{\|B_1A^{\frac{1}{4}+\frac{1}{2}\delta}u\|_{\mathcal{L}(H)}}{\|u\|}\\
			=&\ \sup_{v\in \dot{H}^{-\frac{1}{2}-\delta}}\frac{\|B_1v\|_{\mathcal{L}(H)}}{\|v\|_{-\frac{1}{2}-\delta}} = \|B_1\|_{\mathcal{L}(\dot{H}^{-\frac{1}{2}-\delta},\mathcal{L}(H))}\le C_G,
		\end{aligned}
	\end{equation*}
	and similarly we obtain for $u\in \dot{H}^{-\frac{1}{2}-\delta}$
	\begin{equation*}
		\|Dg(u)A^{\frac{1}{4}+\frac{1}{2}\delta}\|_{\mathcal{L}(H,\mathcal{L}(H))} = \|Dg(u)\|_{\mathcal{L}(\dot{H}^{-\frac{1}{2}-\delta},\mathcal{L}(H))} \le C_G,
	\end{equation*}
	which implies
	\begin{equation*}
		\|DG(u)A^{\frac{1}{4}+\frac{1}{2}\delta}\|_{\mathcal{L}(H,\mathcal{L}(H))} = \|(B_1+Dg(u))A^{\frac{1}{4}+\frac{1}{2}\delta}\|_{\mathcal{L}(H,\mathcal{L}(H))} \le C_G.
	\end{equation*}
	By \eqref{eq2.2.1} and \eqref{eq2.2.00} we have for any $v,w\in H$
	\begin{equation*}
		\begin{aligned}
			\|DG(u)v\|_{\mathcal{L}_2(H)} =&\ \|DG(u)A^{\frac{1}{4}+\frac{1}{2}\delta}A^{-\frac{1}{4}-\frac{1}{2}\delta}v\|_{\mathcal{L}_2(H)}\\ 
			\le&\  \|DG(u)A^{\frac{1}{4}+\frac{1}{2}\delta}\|_{\mathcal{L}(H,\mathcal{L}(H))}\|A^{-\frac{1}{4}-\frac{1}{2}\delta}\|_{\mathcal{L}_2(H)}\|v\|\le C\|v\|.
		\end{aligned}
	\end{equation*}
	In a similar way, we get $\|D^2G(u)(v,w)\|_{\mathcal{L}_2(H)} \le C\|v\|\|w\|$.
	
	Moreover, if $u \in \dot{H}^1$, by \eqref{eq2.2.1} and \eqref{eq2.2.00} again, we have
	\begin{equation*}
		\begin{aligned}
			&\ \|A^{\frac{1}{2}}G(u)\|_{\mathcal{L}_2(H)} = \|DG(u)A^{\frac{1}{2}}u\|_{\mathcal{L}_2(H)} = \|DG(u)A^{\frac{1}{4}+\frac{1}{2}\delta}A^{-\frac{1}{4}-\frac{1}{2}\delta}A^{\frac{1}{2}}u\|_{\mathcal{L}_2(H)}\\
			\le&\ \|DG(u)A^{\frac{1}{4}+\frac{1}{2}\delta}\|_{\mathcal{L}(H,\mathcal{L}(H))}\|A^{-\frac{1}{4}-\frac{1}{2}\delta}\|_{\mathcal{L}_2(H)}\|A^{\frac{1}{2}}u\| \le C\|u\|_1.
		\end{aligned}
	\end{equation*}
	Thus the proof is completed.
\end{proof}

\begin{itemize}
	\item[{\bf{\bf(A3)}}] $W(t)$ is an $H$-valued space-time white noise. 
\end{itemize}

Define a filtration $\mathcal{W}_t = \sigma(W(s),0\le s \le t)$ for $0 \le t \le T$. Let $\beta_j(t)$ be a sequence of $i.i.d.$ standard scalar Brownian motion, then $W(t)$ can be formally expressed as
\begin{equation*}
	W(t) = \sum_{j=1}^\infty\beta_j\varphi_j,
\end{equation*}
and this series converges in $L^2(\Omega,\dot{H}^{-\alpha})$ with any $\alpha>\frac{1}{2}$ uniformly for $t \in [0,T]$.

\begin{itemize}
	\item[{\bf(A4)}] The initial value $\xi \in \dot{H}^1$ is a deterministic function.
\end{itemize}

Now we study the regularity estimate for solution to \eqref{eq2.1.1}.

\begin{theorem}\label{t2.2.3}
	Assume {\bf(A1)}-{\bf(A4)}. Then, \eqref{eq2.1.1} possesses a unique mild solution $X(t) \in \dot{H}^1$ adapted to $\mathcal{W}_t$ for $t\in [0,T]$
	\begin{equation}\label{eq2.2.11}
		X(t) = S(t)\xi + \int_0^tS(t-s)F(X(s))ds + \int_0^tS(t-s)G(X(s))dW(s).
	\end{equation}
	Moreover, for any $p \ge 2$, there exists a postive constant $C=C(T,p)$, such that
	\begin{equation}\label{eq2.2.12}
		\mathbb{E}\| X(t)\|_1^p \le C( 1+\| \xi\|_1^p),\quad \forall t \in [0,T].
	\end{equation}
\end{theorem}

\begin{proof} 
	Applying $A^{\frac{1}{2}}$ on both sides of \eqref{eq2.1.1}, we obtain
	\begin{equation*}
		\begin{aligned}
			A^{\frac{1}{2}}X(t) =&\ S(t)A^{\frac{1}{2}}\xi + \int_0^tS(t-s)DF(X(s))A^{\frac{1}{2}}X(s)ds + \int_0^tS(t-s)DG(X(s))A^{\frac{1}{2}}X(s)dW(s).
		\end{aligned}
	\end{equation*}
	Let $Y(t) = A^{\frac{1}{2}}X(t)$, then it satisfies $Y(0) = A^{\frac{1}{2}}\xi$ and
	\begin{equation*}
		dY(t)+AY(t)dt = DF(X(t))Y(t)dt + DG(X(t))Y(t)dW(t).
	\end{equation*}
	For any integer $p \ge 1$, define a functional on $H$ by $\rho(u) = \|u\|^{2p}$. Applying the mild It\^o formula for SPDEs \cite[Theorem 1]{DaPrato2019} to  $\rho(Y(t))$, we have
	\begin{equation}\label{eq2.2.21}
		\begin{aligned}
			\rho(Y(t)) =&\ \rho(Y(0)) + \int_0^t\big(D\rho(Y(s)),-AY(s)+DF(X(s))Y(s)\big)ds\\
			&\ + \int_0^t\big(D\rho(Y(s)),DG(X(s))Y(s)dW(s)\big)\\
			&\ + \frac{1}{2}\int_0^t\!\mathrm{Tr}\{D^2\rho(Y(s))(DG(X(s))Y(s))(DG(X(s))Y(s))^*\}ds.
		\end{aligned}
	\end{equation}
	Here, $(D\rho(Y(s),u) = 2p\|Y(s)\|^{2p-2}(Y(s),u)$ for any $u \in H$, cf.\ \cite[(3.22)]{Gawarecki2011}, therefore we have $(D\rho(Y(s)),-AY(s))\le0$. Moreover
	\begin{equation}\label{eq2.2.22}
		\|D^2\rho(u)\|_{\mathcal{L}(H)} = \|D^2\rho(u)\|_{\mathcal{L}(H\times H,\mathbb{R})} \le 2p(2p-1)\|u\|^{2(p-1)}.
	\end{equation}
	Taking expectation on both sides of \eqref{eq2.2.21}, by \eqref{eq2.2.5}, Lemma \ref{l2.2.1} and \eqref{eq2.2.22}, we obtain
	\begin{equation*}
		\begin{aligned}
			&\ \mathbb{E}\rho(Y(t))\\ 
			\le&\ \mathbb{E}\rho(A^{\frac{1}{2}}\xi) + \mathbb{E}\int_0^t2p\|Y(s)\|^{2p-2}(Y(s),DF(X(s))Y(s))ds\\
			&\ + \frac{1}{2}\mathbb{E}\int_0^t2p(2p-1)\|Y(s)\|^{2(p-1)}\|DG(X(s))Y(s)\|^2_{\mathcal{L}_2(H)}ds\\
			\le&\ \mathbb{E}\rho(A^{\frac{1}{2}}\xi) + \mathbb{E}\int_0^t2p\|Y(s)\|^{2p-2}C_F\|Y(s)\|^2ds\\
			&\ + \frac{1}{2}\mathbb{E}\int_0^t2p(2p-1)\|Y(s)\|^{2(p-1)}C\|Y(s)\|^2ds\\
			\le&\ \mathbb{E}\rho(A^{\frac{1}{2}}\xi) + C\mathbb{E}\int_0^t\|Y(s)\|^{2p}ds.
		\end{aligned}
	\end{equation*}
	From Gronwall inequality, it follows that for any $t \in [0,T]$
	\begin{equation*}
		\mathbb{E}\|X(t)\|_1^{2p} \le C\|\xi\|_1^{2p}.
	\end{equation*}
	This completes the proof.
\end{proof}

By the Sobolev embedding property $\dot{H}^1(\mathcal{D}) \hookrightarrow \mathcal{C}(\bar{\mathcal{D}})$, we have
\begin{equation}\label{eq2.2.11.1}
	\mathbb{E}\| X(t)\|_{V}^p \le C(1+\| \xi \|^p_{1}),\quad \forall t \in [0,T].
\end{equation}

\subsection{Malliavin calculus}
In order to proceed weak convergence analysis, we recall
some properties on Malliavin calculus, cf. \cite[Section 2.4]{Andersson2016} and \cite[Section 2]{Nualart1998}. For any $h \in$ $L^{2}((0,T),H)$, define $W(h) = \int_0^T( h(t),dW(t))$, which is a centered isonormal Gaussian process on $H$, i.e.\
\begin{equation*}
	\mathbb{E}[W(h)W(g)]=(h, g)_{L^{2}((0,T), H)}, \quad \forall h,g \in L^{2}((0,T), H).
\end{equation*}

For any integer $n \ge 1$, let $\mathcal{C}_{p}^{\infty}(\mathbb{R}^{n})$ denote the space of all real-valued $\mathcal{C}^{\infty}$-functions defined on $\mathbb{R}^{n}$ with polynomial growth. We define a family of smooth cylindrical random variables with values in $\mathbb{R}$ by
\begin{equation*}
	\begin{aligned}
		\mathcal{S}=&\ \{X=f(W(h_{1}), \cdots, W(h_{n})):  f \in \mathcal{C}_{p}^{\infty}(\mathbb{R}^{n}),\\
		&\qquad h_{1}, \cdots, h_{n} \in L^{2}((0,T), H),\ n=1,2,\cdots\},
	\end{aligned}
\end{equation*}
and a family of counterparts with values in $H$ by
\begin{equation*}
	\mathcal{S}(H)=\{F=\sum_{i=1}^{m} X_{i} \varphi_{i}: X_{i} \in \mathcal{S}, m=1,2,\cdots\},
\end{equation*}
where $\varphi_i\in H$ is defined in previous section.

For any $X \in \mathcal{S}$, define its Malliavin derivative by 
\begin{equation*}
	\mathcal{D} X=\sum_{i=1}^{n} \partial_{i} f(W(h_{1}),\cdots, W(h_{n})) h_{i},
\end{equation*}
which is an $L^{2}((0,T), H)$-valued random variable. Equivalently, $(\mathcal{D}_t X)_{t\in(0,T)}$ is an $H$-valued stochastic process. Furthermore, for any $F \in \mathcal{S}(H)$, its Malliavin derivative is given by
\begin{equation*}
	\mathcal{D} F=\sum_{i=1}^{m} \sum_{j=1}^{n} \partial_{j} f_{i}(W(h_{1}), \cdots, W(h_{n}))(\varphi_{i} \otimes h_{j}),
\end{equation*}
where the tensor product $\varphi_{i} \otimes h_{j}$ denotes a bounded bilinear map from $\!H \times L^{2}((0,T), H)\!$ to $\mathbb{R}$ defined by
\begin{equation*}
	(\varphi_{i} \otimes h_{j})(u,g) = ( \varphi_{i},u ) ( h_{j},g)_{L^{2}((0,T), H)},\quad \forall (u,g) \in H \times L^{2}((0,T), H).
\end{equation*}
Noticing  
\begin{equation*}
	H \otimes L^{2}((0,T), H) \simeq \mathcal{L}_2(L^{2}((0,T), H),H) \simeq L^2((0,T),\mathcal{L}_2(H)),
\end{equation*}
then $(\mathcal{D}_{t}F)_{t \in(0,T)}$ can be regarded as an $\mathcal{L}_2(H)$-valued stochastic process.
By $\mathcal{D}_{t}^{u} F$ we denote the Malliavin derivative of $F$ in direction $u \in H$ at time $t$, i.e.
\begin{equation*}
	\mathcal{D}_{t}^{u}F = \mathcal{D}_{t}Fu=\sum_{i=1}^{m} \sum_{j=1}^{n} \partial_{j} f_{i}(W(h_{1}), \cdots, W(h_{n})) ( u, h_{j}(t)) \varphi_{i}.
\end{equation*}

We define a Watanabe-Sobolev space $\mathbb{D}^{1,2}(H)$ as the closure of $\mathcal{S}(H)$ with respect to the norm
\begin{equation*}
	\|F\|^2_{\mathbb{D}^{1,2}(H)}=\mathbb{E}\|F\|^{2}+\mathbb{E}\int_{0}^{T}\|\mathcal{D}_{t} F\|_{\mathcal{L}_2(H)}^{2} dt.
\end{equation*}
The following lemma gives the integration by parts formula in Malliavin sense, which plays an important role in weak error analysis.

\begin{lemma}\label{l2.2.4}\cite[Lemma 2.2]{Andersson2016}
	For any given random variable $F \in \mathbb{D}^{1,2}(H)$ and any $\Phi \in L^{2}((0,T) \times \Omega, \mathcal{L}_2(H))$ adapted to $\mathcal{W}_t$, there holds the integration by parts formula
	\begin{equation*}
		\mathbb{E}( F,\int_0^T\Phi(t)dW(t) ) = \mathbb{E}\int_0^T( \mathcal{D}_tF,\Phi(t) )_{\mathcal{L}_2(H)}dt.
	\end{equation*}
\end{lemma}

\section{Auxiliary equation and regularity estimates}
\label{Sect3}

In this section, we first construct a temporal splitting-up approximation to \eqref{eq2.1.1}, which might be regarded as the exponential Euler method applied to an auxiliary SPDE \cite{Higham2002,Brehier2019,CuiJianbo2019}. Then, we investigate the strong convergence order in time between the exact and auxiliary solutions. Finally, we provide a priori estimates for the corresponding Kolmogorov equation.

\subsection{Construction of auxiliary equation and strong convergence rate}
According to the splitting-up strategy, we first decompose \eqref{eq2.1.1} into an abstract ODE and an SPDE
\begin{align}
	dX(t) =&\ F(X(t))dt,  \label{eq2.4.1}\\
	dX(t)+AX(t)dt =&\ G(X(t))dW(t).   \label{eq2.4.2} 
\end{align}
Let $\Phi_{t}$ with $\Phi_0 = I$ be the solution operator of \eqref{eq2.4.1}, then $X(t)=\Phi_{t}(\xi)$ is a solution satisfiying initial value condition $X(0) = \xi$ of \eqref{eq2.4.1}. According to \cite{Brehier2018}, we have an explicit formula
\begin{equation}\label{eq2.4.3}
	X(t) = \Phi_{t}(\xi) = \frac{\xi}{\sqrt{\xi^2+(1-\xi^2)e^{-2t}}}, \quad t\ge 0.
\end{equation}
The following lemma gives the differentiablility of the phase flow of \eqref{eq2.4.1} with respect to initial value.

\begin{lemma}\label{l2.4.1}\cite[Lemma 4.1]{CuiJianbo2019}
	Assume {\bf(A1)}. Then the phase flow $\Phi_t$ satisfies for all $x \in \mathbb{R}$
	\begin{equation*}\label{eq2.4.4}
		|\Phi_t(x)|\le e^{t}|x|,\quad|D\Phi_t(x)| \le e^{t},\quad |D^2\Phi_t(x)|\le e^{Ct}|x|,
	\end{equation*}
	where $C>0$ is a constant.
\end{lemma}

For any given positive integer $N>0$, let $\tau = \frac{T}{N}$ be time step size and $t_k=k\tau$ $(k=0,1,\cdots,N)$ be the uniform partition of interval $[0,T]$. 
Let $\widetilde{X}_0 = \xi$ and assume $\widetilde{X}_k$ $(k=0,1,\cdots,N)$ is an approximation to \eqref{eq2.1.1} at $t_k$.
In terms of the splitting-up strategy, we will define next approximation at $t_{k+1}$ by first solving \eqref{eq2.4.1} to get $X_{k+1}^* =\ \Phi_\tau(\widetilde{X}_{k})$, and then applying an exponential Euler method to \eqref{eq2.4.2} to compute
\begin{equation}\label{eq2.4.5}
	\widetilde{X}_{k+1} =\ S_{\tau}^{expo}X_{k+1}^*+S_{\tau}^{expo}G(X_{k+1}^*)\Delta W_k,
\end{equation}
where the linear operator $S_{\tau}^{expo} = e^{\tau A}$, $\Delta W_k = W((k+1)\tau)-W(k\tau)$.

Define two auxiliary maps
\begin{equation}\label{eq2.4.6}
	F_{\tau}(\xi) :=
	\frac{\Phi_{\tau}(\xi)-\xi}{\tau},
	\quad
	G_{\tau}(\xi) :=
	G(\Phi_{\tau}(\xi)),\quad \forall \xi \in H.
\end{equation}
Then, the splitting-up approximation in \eqref{eq2.4.5} can be written as
\begin{equation}\label{eq2.4.7}
	\widetilde{X}_{k+1} = S_{\tau}^{expo}\widetilde{X}_k+S_{\tau}^{expo}F_{\tau}(\widetilde{X}_k)\tau + S_{\tau}^{expo}G_{\tau}(\widetilde{X}_k)\Delta W_k.
\end{equation}
Obviously, $\widetilde{X}_{k}$ $(k=0,1,\cdots,N)$ can be regarded as exponential Euler approximation to the solution $X_{\tau}(t)$ at $t=t_k$ of the following auxiliary equation
\begin{equation}\label{eq2.4.8}
	\begin{aligned}
		dX_{\tau}(t)+AX_{\tau}(t)dt =&\ F_{\tau}(X_{\tau}(t))dt+G_{\tau}(X_{\tau}(t))dW(t),\\
		X_{\tau}(0) =&\ \xi.
	\end{aligned}
\end{equation}

The following lemma collect some properties on $F_{\tau}(u)$ together with its Fr\'echet derivatives.

\begin{lemma}\label{l2.4.2.1}\cite[Lemma 4.2]{CuiJianbo2019}
	For any $\tau_0 \in (0,1)$, there exists a positive constant $C=C(\tau_0)$, such that for all $\tau \in (0,\tau_0)$, $x\in \mathbb{R}$, there hold
	\begin{equation*}
		\begin{aligned}
			&|F_{\tau}(x)| \le C(1+|x|^3),\\
			&DF_{\tau}(x) \le e^{C\tau_0},\quad |DF_{\tau}(x)| \le C(1+|x|^2),\\
			&|D^2F_{\tau}(x)| \le C(1+|x|^3),\\
			&|F_{\tau}(x)-F(x)| \le C\tau(1+|x|^5). 
		\end{aligned}
	\end{equation*}
\end{lemma}

From this lemma, it follows that for all $u \in V$ and $v,w \in H$
\begin{equation}\label{eq2.4.9}
	\begin{aligned}
		&\|F_{\tau}(u)\| \le C(1+\|u\|^3_V),\\ 
		&(DF_{\tau}(u)v,v) \le C\|v\|^2,\\ 
		&\|DF_{\tau}(u)v\| \le C(1+\|u\|^2_V)\|v\|,\\ 
		&\|D^2F_{\tau}(u)(v,w)\| \le C(1+\|u\|^3_V)\|v\|\|w\|,\\ 
		&\|F_{\tau}(u)-F(u)\| \le C\tau(1+\|u\|_V^5). 
	\end{aligned}
\end{equation}
We emphasize that $DF_{\tau}$ satisfies a one-sided Lipschitz continuity.

\begin{lemma}\label{l2.4.2}
	For any $\tau_0 \in (0,1)$, there exists a positive constant $C=C(\tau_0)$, such that for all $\tau \in (0,\tau_0)$, $u \in V$ and $v,w \in H$, there hold
	\begin{equation*}
		\begin{aligned}
			\|G_{\tau}(u)\|_{\mathcal{L}(H)} &\le C(1+\|u\|_V),\\
			\|DG_{\tau}(u)v\|_{\mathcal{L}_2(H)} &\le C\|v\|,\\ 
			\|D^2G_{\tau}(u)(v,w)\|_{\mathcal{L}_2(H)} &\le C(1+\|u\|_V)\|v\|\|w\|,\\ 
			\|G_{\tau}(u)-G_{\tau}(v)\|_{\mathcal{L}_2(H)} &\le C\|u-v\|,\\
			\|G_{\tau}(u)-G(u)\|_{\mathcal{L}_2(H)} &\le C\tau(1+\|u\|^3_V). 
		\end{aligned}
	\end{equation*}
\end{lemma}

\begin{proof}
	For any $u \in V$, we have $\Phi_{\tau}(u)\in H$. Then, by \eqref{eq2.4.6}, \eqref{eq2.2.1.2} and Lemma \ref{l2.4.1}, there holds
	\begin{equation*}
		\|G_{\tau}(u)\|_{\mathcal{L}(H)} = \|G(\Phi_{\tau}(u))\|_{\mathcal{L}(H)} \le C(1+\|\Phi_{\tau}(u)\|) \le C(\tau_0)(1+\|u\|_V).
	\end{equation*}
	Similarly, for any $v,w \in H$
	\begin{equation*}
		\begin{aligned}
			&\ \|DG_{\tau}(u)v\|_{\mathcal{L}_2(H)} = \|D(G(\Phi_\tau(u)))v\|_{\mathcal{L}_2(H)}
			= \|DG(\Phi_\tau(u))\big(D\Phi_\tau(u)v\big)\|_{\mathcal{L}_2(H)}\\ 
			\le&\ \|DG(\Phi_\tau(u))\|_{\mathcal{L}_2(H)}\|D\Phi_\tau(u)v\| \le Ce^{C\tau_0}\|v\|.
		\end{aligned}
	\end{equation*}
	Furthermore,
	\begin{equation*}
		\begin{aligned}
			&\ \|D^2G_{\tau}(u)(v,w)\|_{\mathcal{L}_2(H)} = \|D^2(G(\Phi_{\tau}(u)))(v,w)\|_{\mathcal{L}_2(H)}\\ 
			\le&\ \|DG(\Phi_{\tau}(u))\big(D^2\Phi_{\tau}(u)(v,w)\big)\|_{\mathcal{L}_2(H)} + \|D^2G(\Phi_{\tau}(u))\big(D\Phi_{\tau}(u)v,D\Phi_{\tau}(u)w\big)\|_{\mathcal{L}_2(H)}\\
			\le&\ C\|D^2\Phi_{\tau}(u)(v,w)\| + C\|D\Phi_{\tau}(u)v\|\|D\Phi_{\tau}(u)w\| \\
			\le&\ C(\tau_0)(1+\|u\|_V) \|v\|\|w\|.
		\end{aligned}
	\end{equation*}
	By \eqref{eq2.2.1}, \eqref{eq2.4.6}, the mean value theorem and Lemma \ref{l2.2.1}, we have
	\begin{equation*}
		\begin{aligned}
			&\ \|G_{\tau}(u)-G_{\tau}(v)\|_{\mathcal{L}_2(H)}
			= \|G(\Phi_{\tau}(u))-G(\Phi_{\tau}(v))\|_{\mathcal{L}_2(H)}\\
			=&\ \|\int_0^1DG(\theta\Phi_{\tau}(u)+(1-\theta)\Phi_{\tau}(v))d\theta(\Phi_{\tau}(u)-\Phi_{\tau}(v))\|_{\mathcal{L}_2(H)}\\
			\le&\ \|\int_0^1DG(\theta\Phi_{\tau}(u)+(1-\theta)\Phi_{\tau}(v))d\theta\|_{\mathcal{L}_2(H)}\|\Phi_{\tau}(u)-\Phi_{\tau}(v)\|\\
			\le&\ \int_0^1\|DG(\theta\Phi_{\tau}(u)+(1-\theta)\Phi_{\tau}(v))\|_{\mathcal{L}_2(H)}d\theta\cdot C(\tau_0)\|u-v\|
			\le C\|u-v\|.
		\end{aligned}
	\end{equation*}
	Noticing \eqref{eq2.4.9} and in a similar way, we obtain
	\begin{equation*}
		\begin{aligned}
			\|G_{\tau}(u)-G(u)\|_{\mathcal{L}_2(H)}
			=&\ \|G(\Phi_{\tau}(u))-G(u)\|_{\mathcal{L}_2(H)}\\
			\le&\ C\|\Phi_{\tau}(u)-u\| = C\|\tau F_{\tau}(u)\|\le C\tau(1+\|u\|^3_V).
		\end{aligned}
	\end{equation*}
	This completes the proof.
\end{proof}

In \cite{Brehier2020}, it is pointed out that the coefficient functions $F_{\tau}(\cdot)$ and $G_{\tau}(\cdot)$ are globally Lipschitz continuous for any fixed $\tau \in (0,\tau_0)$. According to \cite[Theorem 3.3]{Gawarecki2011}, for any given $\tau\in(0,\tau_0)$, \eqref{eq2.4.8} possesses a unique mild solution $X_{\tau}(t)$. Moreover, by virtue of the uniform estimates for $F_{\tau}(\cdot)$ and $G_{\tau}(\cdot)$ with respect to $\tau\in(0,\tau_0)$ in \eqref{eq2.4.9} and Lemma \ref{l2.4.2}, there exists a constant $C = C(\xi,T,\tau_0,p)>0$ such that
\begin{equation}\label{eq2.4.16}
	\mathbb{E}\|X_{\tau}(t)\|_1^p \le C(1+\|\xi\|_1^p),\quad \forall t \in [0,T], p \ge 2.
\end{equation}
This together with Sobolev embedding property $\dot{H}^1(\mathcal{D}) \hookrightarrow \mathcal{C}(\bar{\mathcal{D}})$ leads to
\begin{equation}\label{eq2.4.17}
	\mathbb{E}\|X_{\tau}(t)\|_V^p \le C(1+\|\xi\|_1^p).
\end{equation}

The main result of this section is the strong convergence rate between $X_{\tau}(t)$ and $X(t)$.

\begin{theorem}\label{t2.4.1}
	Assume {\bf(A1)}-{\bf(A4)}. Then, for any $\tau_0 \in (0,1)$, the auxiliary solution $X_{\tau}(t)$ converges to the solution $X(t)$ of \eqref{eq2.1.1} with order of 1 as $\tau\to 0$, i.e.
	\begin{equation*}
		\mathbb{E}\|X_{\tau}(t)-X(t)\|^2 \le\ C \tau^2, \quad \forall t\in [0,T],\tau \in (0,\tau_0).
	\end{equation*}
	where $C = C(\xi,T,\tau_0)>0$ is a constant.
\end{theorem}

\begin{proof}
	Define an error function $R(t) = X(t)-X_{\tau}(t)$, then $R(0) = 0$ and there holds
	\begin{equation*}
		dR(t) + AR(t)dt = (F(X(t))-F_{\tau}(X_{\tau}(t)))dt + (G(X(t))-G_{\tau}(X_{\tau}(t)))dW(t).
	\end{equation*}
	From the mild It\^o formula, it follows that
	\begin{equation*}
		\begin{aligned}
			\|R(t)\|^2 =&\ 2\int_0^t (R(s),-AR(s)+F(X(s))-F_{\tau}(X_{\tau}(s))) ds \\
			&+ 2\int_0^t\big( R(s),(G(X(s))-G_{\tau}(X_{\tau}(s)))dW(s) \big)\\
			&+ \int_0^t\|G(X(s))-G_{\tau}(X_{\tau}(s))\|^2_{\mathcal{L}_2(H)}ds.
		\end{aligned}
	\end{equation*}
	Taking expectation on both sides, we have
	\begin{equation*}
		\begin{aligned}
			\mathbb{E}\|R(t)\|^2 =&\ 2\mathbb{E}\int_0^t ( R(s),-AR(s)) ds + 2\mathbb{E}\int_0^t ( R(s),F(X(s))-F_{\tau}(X_{\tau}(s))) ds\\
			&+ \mathbb{E}\int_0^t\|G(X(s))-G_{\tau}(X_{\tau}(s))\|^2_{\mathcal{L}_2(H)}ds\\
			:=&\ I_1+I_2+I_3.
		\end{aligned}
	\end{equation*}
	
	Obviously, we get
	\begin{equation*}
		I_1 = -2\mathbb{E}\int_0^t (A^{\frac{1}{2}}R(s),A^{\frac{1}{2}}R(s)) ds \le 0. 
	\end{equation*}
	By \eqref{eq2.4.9} and \eqref{eq2.2.11.1}, we have
	\begin{equation*}
		\begin{aligned}
			I_2 =&\ 2\mathbb{E}\!\int_0^t\! (R(s),F(X(s))-F_{\tau}(X(s))) ds\! +\! 2\mathbb{E}\!\int_0^t\!( R(s),F_{\tau}(X(s))-F_{\tau}(X_{\tau}(s))) ds\\
			\le&\ \mathbb{E}\int_0^t\|R(s)\|^2ds + \mathbb{E}\int_0^t \|F(X(s))-F_{\tau}(X(s))\|^2 ds\\
			&+ 2\mathbb{E}\int_0^t \big(R(s),\int_0^1DF_{\tau}(\theta X(s)+(1-\theta)X_{\tau}(s))R(s)d\theta\big) ds\\
			=&\ \mathbb{E}\int_0^t\|R(s)\|^2ds + \mathbb{E}\int_0^t \|F(X(s))-F_{\tau}(X(s))\|^2 ds\\
			&+ 2\mathbb{E}\int_0^t \int_0^1\big(R(s),DF_{\tau}(\theta X(s)+(1-\theta)X_{\tau}(s))R(s)\big) d\theta ds\\
			\le&\ \mathbb{E}\int_0^t\|R(s)\|^2ds +C\tau^2 \mathbb{E}\int_0^t(1+\|X(s)\|_V^{5})^2ds + \!C\mathbb{E}\int_0^t\|R(s)\|^2ds\\
			\le&\  CT(1+\|\xi\|_1^{5})^2\tau^2 + C\int_0^t\mathbb{E}\|R(s)\|^2ds.
		\end{aligned}
	\end{equation*}
	In a similar way and by Lemma \ref{l2.4.2}, we have
	\begin{equation*}
		\begin{aligned}
			I_3 \le&\ \mathbb{E}\int_0^t\|G(X(s))-G_{\tau}(X(s))\|^2_{\mathcal{L}_2(H)} + \|G_{\tau}(X(s))-G_{\tau}(X_{\tau}(s))\|^2_{\mathcal{L}_2(H)}ds\\
			\le&\ C\tau^2\mathbb{E}\int_0^t(1+\|X(s)\|^3_V)^2ds + C\mathbb{E}\int_0^t\|R(s)\|^2ds\\
			\le&\ CT(1+\|\xi\|^3_1)^2\tau^2 + C\int_0^t\mathbb{E}\|R(s)\|^2ds.
		\end{aligned}
	\end{equation*}
	Above all, we have proved
	\begin{equation*}
		\mathbb{E}\|R(t)\|^2 \le C\tau^2 + C\int_0^t\mathbb{E}\|R(s)\|^2ds.
	\end{equation*}
	By Gronwall's inequality, we get $\mathbb{E}\|R(t)\|^2 \le C\tau^2$,
	where $C = C(\xi,T,\tau_0)>0$ is a constant.
\end{proof}

\subsection{Kolmogorov equation and a priori estimates}
For any given $\phi \in \mathcal{C}_b^2(H)$, define $U_{\tau}(t,\xi) = \mathbb{E}[\phi(X_{\tau}(t,\xi))]$.
According to \cite[Chap 9]{DaPrato2014}, $U_{\tau}(t,\xi)$ is a solution of the initial value problem of Kolmogorov equation associated with \eqref{eq2.4.8}
\begin{equation}\label{eq2.4.13}
	\begin{aligned}
		\frac{\partial U_{\tau}(t,\xi)}{\partial t} =&\ \mathcal{L}_{\tau}U_{\tau}(t,\xi)\\ 
		:=&\ (-A\xi+ F_{\tau}(\xi), DU_{\tau}(t,\xi))
		+ \frac{1}{2}\mathrm{Tr}\{D^2U_{\tau}(t,\xi)G_{\tau}(\xi)G_{\tau}^*(\xi)\} , \\
		U_{\tau}(0,\xi) =&\ \phi(\xi).
	\end{aligned}
\end{equation}
Now, we study the regularity of $U_{\tau}(t,\xi)$. 
For any given $t\in [0,T]$ and $\xi \in \dot{H}^1$, noticing that 
\begin{equation*}
	D^2U_{\tau}(t,\xi) \in \mathcal{L}(H\times H, \mathbb{R})\simeq\mathcal{L}(H,\mathcal{L}(H,\mathbb{R}))\simeq\mathcal{L}(H,H),
\end{equation*}
then $D^2U_{\tau}(t,\xi)$ can be regarded as a bounded linear map from $H$ to $H$.

\begin{lemma}\label{l2.4.3}
	Assume {\bf(A1)}-{\bf(A4)}. For any $\tau_0 \in (0,1)$ and $\phi \in \mathcal{C}_b^2(H)$, there exists a positive constant $C = C(T,\tau_0,\phi)$ such that for all $\tau \in (0,\tau_0)$, $\xi\in \dot{H}^1$ and $t \in [0,T]$
	\begin{equation*}
		\begin{aligned}
			\|DU_{\tau}(t,\xi)\| \le&\ C, \\
			\|D^2U_{\tau}(t,\xi)\|_{\mathcal{L}(H)} \le&\ C(1+\|\xi\|_1^{3}).
		\end{aligned}
	\end{equation*}
\end{lemma}

\begin{proof}
	Differentiating $U_{\tau}(t,\xi)$ with respect to $\xi$ along direction $y \in H$, according to \cite[Theorem 9.8]{DaPrato2014}, we obtain
	\begin{equation}\label{eq2.4.13.1}
		DU_{\tau}(t,\xi)y = \mathbb{E}[D\phi(X_{\tau}(t,\xi))\eta^{y,\xi}(t)],
	\end{equation}
	where $\eta^{y,\xi}(t) = DX_{\tau}(t,\xi)y$ satisfies $\eta^{y,\xi}(0) =y$ and
	\begin{equation}\label{eq2.4.13.6}
		d\eta^{y,\xi}(t) =\ -A\eta^{y,\xi}(t)dt+DF_{\tau}(X_{\tau}(t))\eta^{y,\xi}(t)dt + DG_{\tau}(X_{\tau}(t))\eta^{y,\xi}(t)dW(t).
	\end{equation}
	For simplifying notations, we will write $DF_{\tau}(X_{\tau}(t))$ and $DG_{\tau}(X_{\tau}(t))$ as $DF_{\tau}(t)$ and $DG_{\tau}(t)$ when no confusion occurs, respectively. 
	
	Let $p \ge 1$ be an integer and define a functional on $H$ by $\rho(u) = \|u\|^{2p}$. Applying the mild It\^o formula to $\rho(\eta^{y,\xi}(t))$, we have
	\begin{equation*}
		\begin{aligned}
			\rho(\eta^{y,\xi}(t)) =&\ \rho(\eta^{y,\xi}(0)) + \int_0^t\big(D\rho(\eta^{y,\xi}(s)),-A\eta^{y,\xi}(s)+DF_{\tau}(s)\eta^{y,\xi}(s)\big)ds\\
			&+ \int_0^t\big(D\rho(\eta^{y,\xi}(s)),DG_{\tau}(s)\eta^{y,\xi}(s)dW(s)\big)\\
			&+ \frac{1}{2}\int_0^t\mathrm{Tr}\{D^2\rho(\eta^{y,\xi}(s))(DG_{\tau}(s)\eta^{y,\xi}(s))(DG_{\tau}(s)\eta^{y,\xi}(s))^*\}ds.
		\end{aligned}
	\end{equation*}
	Here, $(D\rho(\eta^{y,\xi}(s)),u) = 2p\|\eta^{y,\xi}(s)\|^{2p-2}(\eta^{y,\xi}(s),u)$ for any $u \in H$, cf.\ \cite[(3.22)]{Gawarecki2011}, therefore we have $(D\rho(\eta^{y,\xi}(s)),-A\eta^{y,\xi}(s))\le0$.
	Applying \eqref{eq2.2.0}, \eqref{eq2.2.1}, \eqref{eq2.2.22}, \eqref{eq2.4.9}, and Lemma \ref{l2.4.2}, we get
	\begin{equation*}
		\begin{aligned}
			&\ \mathbb{E}\|\eta^{y,\xi}(t)\|^{2p} = \mathbb{E}\rho(\eta^{y,\xi}(t))\\ 
			\le&\ \mathbb{E}\rho(h) + \mathbb{E}\int_0^t2p\|\eta^{y,\xi}(s)\|^{2p-2}(\eta^{y,\xi}(s),DF_{\tau}(s)\eta^{y,\xi}(s))ds\\
			&+\! \frac{1}{2}\mathbb{E}\!\int_0^t\!\|D^2\rho(\eta^{y,\xi}(s))(DG_{\tau}(s)\eta^{y,\xi}(s))(DG_{\tau}(s)\eta^{y,\xi}(s))^*\|_{\mathcal{L}_1(H)}ds\\
			\le&\ \mathbb{E}\rho(h) + \mathbb{E}\int_0^t2p\|\eta^{y,\xi}(s)\|^{2p-2}\cdot C\|\eta^{y,\xi}(s)\|^2ds\\
			&+\! \frac{1}{2}\mathbb{E}\!\int_0^t\!\|D^2\rho(\eta^{y,\xi}(s))\|_{\mathcal{L}(H)}\!\|(DG_{\tau}(s)\eta^{y,\xi}(s))(DG_{\tau}(s)\eta^{y,\xi}(s))^*\|_{\mathcal{L}_1(H)}ds\\
			\le&\ \mathbb{E}\rho(h) + \mathbb{E}\int_0^tC\|\eta^{y,\xi}(s)\|^{2p}ds\\ 
			&\ + \frac{1}{2}\mathbb{E}\int_0^t2p(2p-1)\|\eta^{y,\xi}(s)\|^{2(p-1)}\|DG_{\tau}(s)\eta^{y,\xi}(s)\|^2_{\mathcal{L}_2(H)}ds\\
			\le&\ \mathbb{E}\rho(h) + \mathbb{E}\int_0^tC\|\eta^{y,\xi}(s)\|^{2p}ds+ \frac{1}{2}\mathbb{E}\int_0^t2Cp(2p-1)\|\eta^{y,\xi}(s)\|^{2(p-1)}\|\eta^{y,\xi}(s)\|^2ds\\
			\le&\ \mathbb{E}\rho(h) + C\mathbb{E}\int_0^t\|\eta^{y,\xi}(s)\|^{2p}ds.
		\end{aligned}
	\end{equation*}
	By Gronwall's inequality, we have
	\begin{equation}\label{eq2.4.13.2}
		\mathbb{E}\|\eta^{y,\xi}(t)\|^{2p} \le C\|y\|^{2p},
	\end{equation}
	where $C = C(T,\tau_0)>0$ is a constant.
	
	By \eqref{eq2.4.13.1}, H\"older inequality and \eqref{eq2.4.13.2} with $p=1$, we obtain
	\begin{equation*}
		\begin{aligned}
			&\ |DU_{\tau}(t,\xi)y| = |\mathbb{E}(D\phi(X_{\tau}(t,\xi)),\eta^{y,\xi}(t))|\\
			\le&\ C|\phi|_{\mathcal{C}_b^1(H)}\mathbb{E}\|\eta^{y,\xi}(t)\| \le C|\phi|_{\mathcal{C}_b^1(H)}(\mathbb{E}\|\eta^{y,\xi}(t)\|^2)^{\frac{1}{2}}
			\le C\|y\|,
		\end{aligned}
	\end{equation*}
	which implies the first estimate.
	
	Next we investigate the second inequality. Differentiating $U_{\tau}(t,\xi)$ twice with respect to $\xi$ and by \cite[Theorem 9.9]{DaPrato2014}, we obtain for $y,z \in H$
	\begin{equation}\label{eq2.4.13.4}
		D^2U_{\tau}(t,\xi)(y,z) =\ \mathbb{E}[D^2\phi(X_{\tau}(t,\xi))\big(\eta^{y,\xi}(t),\eta^{z,\xi}(t)\big)]+ \mathbb{E}[D\phi(X_{\tau}(t,\xi))\zeta^{y,z,\xi}(t)],
	\end{equation}
	where $\zeta^{y,z,\xi}(t) = D^2X_{\tau}(t,\xi)(y,z) = D\eta^{y,\xi}(t)z$ satisfies $\zeta^{y,z,\xi}(0) = 0$ and 
	\begin{equation}\label{eq2.4.13.7}
		\begin{aligned}
			d\zeta^{y,z,\xi}(t) =&\ \big(-A\zeta^{y,z,\xi}(t)+D^2F_{\tau}(t)\big(\eta^{y,\xi}(t),\eta^{z,\xi}(t)\big) + DF_{\tau}(t)\zeta^{y,z,\xi}(t)\big) dt\\
			&+ \big(D^2G_{\tau}(t)\big(\eta^{y,\xi}(t),\eta^{z,\xi}(t)\big) + DG_{\tau}(t)\zeta^{y,z,\xi}(t)\big)dW(t)\\
			:=&\ (a_{11}(t)+a_{12}(t)+a_{13}(t))dt + (a_{21}(t)+a_{22}(t))dW(t).
		\end{aligned}
	\end{equation}
	Obviously, the estimates for $a_{11}$, $a_{13}$ and $a_{22}$ follows from previous analysis with $p=1$.
	According to \eqref{eq2.4.9}, Lemma \ref{l2.4.2} and H\"older inequality, we have
	\begin{equation*}
		\begin{aligned}
			&\ \mathbb{E}\|a_{12}(t)\|^2
			\le C\mathbb{E}\big((1+\|X_{\tau}(t)\|_V^3)^2\|\eta^{y,\xi}(t)\|^2\|\eta^{z,\xi}(t)\|^2\big)\\
			\le&\ C\big(\mathbb{E}(1+\|X_{\tau}(t)\|_V^3)^4\big)^{\frac{1}{2}}\big(\mathbb{E}\|\eta^{y,\xi}(t)\|^8\big)^{\frac{1}{4}}\big(\mathbb{E}\|\eta^{z,\xi}(t)\|^8\big)^{\frac{1}{4}}
			\le C(1+\|\xi\|_1^6)\|y\|^2\|z\|^2
		\end{aligned}
	\end{equation*}
	and
	\begin{equation*}
		\mathbb{E}\|a_{21}(t)\|_{\mathcal{L}_2(H)}^2
		\le C\mathbb{E}\big((1+\|X_{\tau}(t)\|_V)^2\|\eta^{y,\xi}(t)\|^2\|\eta^{z,\xi}(t)\|^2\big)
		\le C(1+\|\xi\|^2_1)\|y\|^2\|z\|^2.
	\end{equation*}
	For any $u \in H$, define $\rho(u) = \|u\|^2$ and apply the mild It\^o formula to $\rho(\zeta^{y,z,\xi}(t))$, we obtain 
	\begin{equation*}
		\begin{aligned}
			\rho(\zeta^{y,z,\xi}(t)) =&\ \rho(\zeta^{y,z,\xi}(0)) + 2\int_0^t (a_{11}(s)+a_{12}(s)+a_{13}(s),\zeta^{y,z,\xi}(s)) ds\\
			&+ 2\int_0^t (\zeta^{y,z,\xi}(s),(a_{21}(s)+a_{22}(s))dW(s))\\
			&+ \frac{1}{2}\int_0^t\mathrm{Tr}\{D^2\rho(\zeta^{y,z,\xi}(s))(a_{21}(s)+a_{22}(s))(a_{21}(s)+a_{22}(s))^*\}ds.
		\end{aligned}
	\end{equation*}
	Inequality \eqref{eq2.2.22} implies $\|D^2\rho(u)\|_{\mathcal{L}(H)} \le 2$, then by \eqref{eq2.2.1} and \eqref{eq2.4.13.2} we have
	\begin{equation*}
		\begin{aligned}
			&\ \mathbb{E}\|\zeta^{y,z,\xi}(t)\|^2\\ 
			\le&\ 2\mathbb{E}\int_0^t (a_{11}(s)+a_{12}(s)+a_{13}(s),\zeta^{y,z,\xi}(s)) ds\\
			&\ + \mathbb{E}\int_0^t\|(a_{21}(s)+a_{22}(s))(a_{21}(s)+a_{22}(s))^*\|_{\mathcal{L}_1(H)}ds\\
			\le&\ 2\mathbb{E}\int_0^t (a_{11}(s)+a_{13}(s),\zeta^{y,z,\xi}(s))ds + \mathbb{E}\int_0^t \|a_{12}(s)\|^2ds + \mathbb{E}\int_0^t\|\zeta^{y,z,\xi}(s))\|^2ds\\
			&\ + 2\mathbb{E}\int_0^t\|a_{21}(s)\|_{\mathcal{L}_2(H)}^2+\|a_{22}(s)\|_{\mathcal{L}_2(H)}^2ds\\
			\le&\ C(1+\|\xi\|_1^6)\|y\|^2\|z\|^2 + C(1+\|\xi\|^2_1)\|y\|^2\|z\|^2 + C\mathbb{E}\int_0^t\|\zeta^{y,z,\xi}(s)\|^2ds.
		\end{aligned}
	\end{equation*}
	From Gronwall's inequality, it follows that
	\begin{equation}\label{eq2.4.13.5}
		\mathbb{E}\|\zeta^{y,z,\xi}(t)\|^2 \le C(1+\|\xi\|^6_1)\|y\|^2\|z\|^2.
	\end{equation}
	By \eqref{eq2.4.13.2}, \eqref{eq2.4.13.5} and H\"older inequality, we have
	\begin{equation*}
		\begin{aligned}
			&\ |D^2U_{\tau}(t,\xi)(y,z)| \\
			\le&\ |\mathbb{E}[D^2\phi(X_{\tau}(t,\xi))\big(\eta^{y,\xi}(t),\eta^{z,\xi}(t)\big)]| + |\mathbb{E}[D\phi(X_{\tau}(t,\xi))\zeta^{y,z,\xi}(t)]|\\
			\le&\ |\phi|_{\mathcal{C}_b^2(H)}\mathbb{E}\big[\|\eta^{y,\xi}(t)\|\cdot\|\eta^{z,\xi}(t)\|\big] + C|\phi|_{\mathcal{C}_b^1(H)}\mathbb{E}\|\zeta^{y,z,\xi}(t)\|\\
			\le&\ |\phi|_{\mathcal{C}_b^2(H)}(\mathbb{E}\|\eta^{y,\xi}(t)\|^2)^{\frac{1}{2}}(\mathbb{E}\|\eta^{z,\xi}(t)\|^2)^{\frac{1}{2}} + C|\phi|_{\mathcal{C}_b^1(H)}(\mathbb{E}\|\zeta^{y,z,\xi}(t)\|^2)^{\frac{1}{2}}\\
			\le&\ |\phi|_{\mathcal{C}_b^2(H)}\|y\|\|z\| + C|\phi|_{\mathcal{C}_b^1(H)}(1+\|\xi\|_1^{3})\|y\|\|z\|\\
			\le&\ C(1+\|\xi\|_1^{3})\|y\|\|z\|,
		\end{aligned}
	\end{equation*}
	where $C = C(T,\tau_0,\phi)>0$ is a constant.
\end{proof}

\section{Finite element approximation and regularity estimates}
\label{Sect4}

\subsection{Semi-disretized finite element approximation}
Let $h>0$ and $T_h$ be a family of regular partition of $\mathcal{D}$ with mesh size $h$. Let $V_h^0 \subset \dot{H}^1$ denote a family of finite element spaces of continuous piecewise linear functions with vanishing value at $\partial \mathcal{D}$. By $P^{h}:H \to V_h^0$ we denote the $L^2$-orthogonal projection operator.
Let $A_{h}: V_{h}^0 \rightarrow V_{h}^0$ be the discrete Laplacian operator defined by
\begin{equation*}
	( A_{h} \psi, \chi)=(\nabla \psi, \nabla \chi), \quad \forall \psi, \chi \in V_{h}^0.
\end{equation*}
Then take $\chi=\psi$, we get
\begin{equation}\label{eq2.3.0}
	\|A_{h}^{\frac{1}{2}} \psi\|^2=\|\nabla \psi\|^2=\|A^{\frac{1}{2}} \psi\|^2=\|\psi\|^2_{1}.
\end{equation}
We will often use the equivalence of the two norms with $\alpha \in [-\frac{1}{2},\frac{1}{2}]$, cf.\ \cite{Andersson2016}
\begin{equation}\label{eq2.3.1}
	C_1\| A_{h}^\alpha v^{h} \| \le \| A^\alpha v^{h} \| \le C_2\| A_{h}^\alpha v^{h} \| ,\quad \forall v^{h} \in V_{h}^0,
\end{equation}
where $C_i = C_i(\alpha)$ is a positive constant for $i = 1,2$. 

Let $S^{h}(t)$ be an analytic semigroup generated by $A_{h}$ and define a Ritz projection operator $R^{h}:\dot{H}^1\to V_{h}^0$ by
\begin{equation*}
	( \nabla R^{h}\psi,\nabla\chi ) = ( \nabla \psi,\nabla\chi ),\quad \forall \psi \in \dot{H}^1,\chi \in V_{h}^0.
\end{equation*}
Then, there exists an $h_0\in (0,1)$ such that for $h \in (0,h_0)$ there holds, cf.\ \cite{Andersson2016} 
\begin{equation}\label{eq2.3.2}
	\begin{aligned}
		&\|A^{\frac{s}{2}}(I-R^{h})A^{-\frac{r}{2}}\|_{\mathcal{L}(H)} \le Ch^{r-s}, \quad 0\le s \le 1 \le r \le2,\\
		&\|A^{\frac{s}{2}}(I-P^{h})A^{-\frac{r}{2}}\|_{\mathcal{L}(H)} \le Ch^{r-s}, \quad 0\le s \le 1,\ 0\le s\le r \le2,\\
		&\| A_{h}^\alpha S^{h}(t)P^{h} \|_{\mathcal{L}(H)}\le Ct^{-\alpha} , \quad 0 \le \alpha,\  0<t,
	\end{aligned}
\end{equation}
where $C=C(h_0,s,r,\alpha)$ is a positive constant.

The semi-discretized finite element approximation to \eqref{eq2.1.1} is defined as seeking $X^{h}(t)\in V_{h}^0$ such that
\begin{equation}\label{eq2.3.5}
	\begin{aligned}
		dX^{h}(t) + A_{h}X^{h}(t)dt =&\ P^{h}F(X^{h}(t))dt + P^{h}G(X^{h}(t))dW(t),\ t\in (0,T],\\
		X^{h}(0) =&\ P^{h} \xi.
	\end{aligned}
\end{equation}
Its mild solution $X^{h}(t)=X^{h}(t,\xi)$ is given by
\begin{equation}\label{eq2.3.6}
	X^{h}(t) = S^{h}(t)P^{h}\xi + \int_0^tS^{h}(t-s)P^{h}F(X^{h}(s))ds + \int_0^tS^{h}(t-s)P^{h}G(X^{h}(s))dW(s).
\end{equation}

We can apply the same argument in proving Theorem \ref{t2.2.3} to obtain a prior estimate for $X^h(t)$ in $\dot{H}^1$-norm, which is presented in next lemma and the proof is omitted.

\begin{lemma}\label{l2.3.1}
	Assume {\bf(A1)}-{\bf(A4)}. Then, for any $p \ge 2$, there exists a postive constant $C=C(\xi,T,p,h_0)$ such that for all $h \in (0,h_0)$ and $t \in [0,T]$
	\begin{equation*}
		\mathbb{E}\| X^{h}(t)\|_1^p +
		\mathbb{E}\| X^{h}(t)\|_{V}^p \le C.
	\end{equation*}
\end{lemma}

For any $\phi \in C_b^2(H)$, define $U^{h}(t,\xi) = \mathbb{E}[\phi(X^{h}(t,\xi))]$ with $U^{h}(0,\xi) = \phi(\xi)$, which is a solution of an initial value problem of Kolmogorov equation associated with \eqref{eq2.3.5}, cf.\ \cite[Chap 4]{Andersson2016}
\begin{equation}\label{eq2.3.9}
	\begin{aligned}
		&\ \frac{\partial U^{h}(t,\xi)}{\partial t} =\ \mathcal{L}^{h} U^{h}(t,\xi)\\ 
		=&\ ( -A_{h}\xi+P^{h}F(\xi),DU^{h}(t,\xi)) 
		+ \frac{1}{2}\mathrm{Tr}\{D^2U^{h}(t,\xi)P^{h}G(\xi)(P^{h}G(\xi))^*\}.
	\end{aligned}
\end{equation}

\subsection{Regularity estimate}
In this section, we investigate the regularity of the solution $X^{h}(t)$. 

\begin{theorem}\label{th2.3.1}
	Assume {\bf(A1)}-{\bf(A4)}. Then for any $\beta \in [0,\frac{1}{2})$, there exists a positive constant $C=C(\xi,T,h_0,\beta)$ such that for all $h \in(0,h_0)$, there holds
	\begin{equation*}
		\mathbb{E}\|A_{h}^{\beta}\mathcal{D}_sX^{h}(t)\|_{\mathcal{L}(H)}^2 \le C, \quad \forall 0\le s < t \le T.
	\end{equation*}
\end{theorem}

\begin{proof}
	We first prove the case of $\beta = 0$. For any given $u \in H$, let $\mathcal{D}_s^uX^{h}(t,\xi) = \mathcal{D}_sX^{h}(t)u$ denote the Malliavin derivative of $X^{h}(t,\xi)$ in direction $u$. Taking Malliavin derivative in direction $u$ on both sides of \eqref{eq2.3.6}, we get for $0\le s \le t \le T$
	\begin{equation}\label{eq2.3.10}
		\begin{aligned}
			\mathcal{D}_s^uX^{h}(t) =&\ \int_s^tS^{h}(t-r)P^{h}DF(X^{h}(r)) \mathcal{D}_s^uX^{h}(r)dr \ +\\
			&\ S^{h}(t-s)P^{h}G(X^{h}(s))u + \int_s^t S^{h}(t-r)P^{h}DG(X^{h}(r)) \mathcal{D}_s^uX^{h}(r)dW(r),
		\end{aligned}
	\end{equation}
	and for other $s$ and $t$, $\mathcal{D}_s^uX^{h}(t) = 0$.
	According to \cite[Chapter 5]{Kruse2014}, $\mathcal{D}_s^uX^{h}(t)$ is a mild solution to the following equation for $0\le s \le t \le T$
	\begin{equation*}
		\begin{aligned}
			d\mathcal{D}_s^uX^{h}(t) =&\ -A_{h}\mathcal{D}_s^uX^{h}(t)dt +  P^{h}DF(X^{h}(t))\mathcal{D}_s^uX^{h}(t)dt\\ 
			&\ + P^{h}DG(X^{h}(t))\mathcal{D}_s^uX^{h}(t)dW(t).
		\end{aligned}
	\end{equation*}
	Let $t=s$ in \eqref{eq2.3.10}, then we get $\mathcal{D}_s^uX^{h}(s) = P^{h}G(X^{h}(s))u$. For later use, we will directly estimate $\mathbb{E}\|\mathcal{D}_s^uX^{h}(t)\|^4$ instead of $\mathbb{E}\|\mathcal{D}_s^uX^{h}(t)\|^2$.  Define a functional $\rho(u)=\|u\|^4$ on $H$ and apply the mild It\^o formula, we obtain
	\begin{equation*}
		\begin{aligned}
			&\ \|\mathcal{D}_s^uX^{h}(t)\|^4 = \rho(\mathcal{D}_s^uX^{h}(t))\\ =&\ \|P^{h}G(X^{h}(s))u\|^4 +  \int_s^t\big(D\rho(\mathcal{D}_s^uX^{h}(r)),(-A_{h}+P^{h}DF(X^{h}(r)))\mathcal{D}_s^uX^{h}(r)\big)dr\\
			&\ + \int_s^t\big(D\rho(\mathcal{D}_s^uX^{h}(r)),P^{h}DG(X^{h}(r))\mathcal{D}_s^uX^{h}(r)dW(r)\big)\\
			&\ +\frac{1}{2}\int_s^t\!\!\mathrm{Tr}\{D^2\rho(\mathcal{D}_s^uX^{h}(r))(P^{h}DG(X^{h}(r))\mathcal{D}_s^uX^{h}(r))(P^{h}DG(X^{h}(r))\mathcal{D}_s^uX^{h}(r))^*\}dr.
		\end{aligned}
	\end{equation*}
	Here $(D\rho(\mathcal{D}_s^uX^{h}(r)),u) = 4\|\mathcal{D}_s^uX^{h}(r)\|^2(\mathcal{D}_s^uX^{h}(r),u)$ for any $u \in H$, therefore
	\begin{equation*}
		(D\rho(\mathcal{D}_s^uX^{h}(r)),-A_{h}\mathcal{D}_s^uX^{h}(r))\le0.
	\end{equation*}  
	By \eqref{eq2.2.5}, \eqref{eq2.2.1.2}, \eqref{eq2.2.22} and Lemma \ref{l2.3.1}, we have
	\begin{equation*}
		\begin{aligned}
			\mathbb{E}\|\mathcal{D}_s^uX^{h}(t)\|^4\le&\ \mathbb{E}\|P^{h}G(X^{h}(s))u\|^4 + 4C\mathbb{E}\int_s^t\|\mathcal{D}_s^uX^{h}(r)\|^4dr\\
			&\ +6\mathbb{E}\int_s^t\|\mathcal{D}_s^uX^{h}(r)\|^2\cdot\|P^{h}DG(X^{h}(r))\mathcal{D}_s^uX^{h}(r)\|^2_{\mathcal{L}_2(H)}dr\\
			\le&\ C\mathbb{E}(1+\|X^{h}(s)\|)^4\|u\|^4 + C\mathbb{E}\int_s^t\|\mathcal{D}_s^uX^{h}(r)\|^4dr\\
			\le&\ C\|u\|^4 + C\mathbb{E}\int_s^t\|\mathcal{D}_s^uX^{h}(r)\|^4dr,
		\end{aligned}
	\end{equation*}
	By Gronwall's inequality, we have
	\begin{equation}\label{eq2.3.12}
		\mathbb{E}\|\mathcal{D}_s^uX^{h}(t)\|^4 \le C\|u\|^4,
	\end{equation}
	where $C = C(\xi,T,h_0)>0$ is a constant. By H\"older inequality, we can deduce that
	\begin{equation*}
		\mathbb{E}\|\mathcal{D}_s^uX^{h}(t)\|^2 \le (\mathbb{E}\|\mathcal{D}_s^uX^{h}(t)\|^4)^{\frac{1}{2}} \le C\|u\|^2.
	\end{equation*}
	
	Now we investigate the case of $\beta \in (0,\frac{1}{2})$. Applying $A_{h}^{\beta}$ to both sides of \eqref{eq2.3.10}, we obtain
	\begin{equation}\label{eq2.3.11}
		\begin{aligned}
			A_{h}^{\beta}\mathcal{D}_s^uX^{h}(t) =&\ A_{h}^{\beta}S^{h}(t-s)P^{h}G(X^{h}(s))u +\!\! \int_s^t\!\!A_{h}^{\beta}S^{h}(t-r)P^{h}DF(X^{h}(r)) \mathcal{D}_s^uX^{h}(r)dr\\
			&+ \int_s^t A_{h}^{\beta}S^{h}(t-r)P^{h}DG(X^{h}(r)) \mathcal{D}_s^uX^{h}(r)dW(r).
		\end{aligned}
	\end{equation}
	By H\"older inequality and It\^o isometry, we obtain
	\begin{equation*}
		\begin{aligned}
			\mathbb{E}\|A_{h}^{\beta}\mathcal{D}_s^uX^{h}(t)\|^2 
			\le&\ 3\mathbb{E}\|A_{h}^{\beta}S^{h}(t-s)P^{h}G(X^{h}(s))u\|^2\\
			&\ +3(t-s)\mathbb{E}\int_s^t\|A_{h}^{\beta}S^{h}(t-r)P^{h}DF(X^{h}(r)) \mathcal{D}_s^uX^{h}(r)\|^2dr\\
			&\ +3\mathbb{E}\int_s^t \|A_{h}^{\beta}S^{h}(t-r)P^{h}DG(X^{h}(r)) \mathcal{D}_s^uX^{h}(r)\|_{\mathcal{L}_2(H)}^2dr\\
			=&\ I+II+III.
		\end{aligned}
	\end{equation*}
	According to \eqref{eq2.3.1} and the Sobolev embedding property $\dot{H}^1\hookrightarrow\dot{H}^{2\beta}$, we have
	\begin{equation*}
		\begin{aligned}
			&\ \|A_{h}^{\beta}P^{h}G(X^{h}(s))u\| \le C\|A^{\beta}P^{h}G(X^{h}(s))u\| = C\|P^{h}G(X^{h}(s))u\|_{2\beta}\\ 
			\le&\ C\|P^{h}G(X^{h}(s))u\|_{1} = C\|A^{\frac{1}{2}}P^{h}G(X^{h}(s))u\| \le C\|A_{h}^{\frac{1}{2}}P^{h}G(X^{h}(s))u\|.
		\end{aligned}
	\end{equation*}
	Therefore, by Lemma \ref{l2.2.1} and \ref{l2.3.1}, we get
	\begin{equation*}
		\begin{aligned}
			I =&\ 3\mathbb{E}\|S^{h}(t-s)A_{h}^{\beta}P^{h}G(X^{h}(s))u\|^2 \\
			\le&\ 3\mathbb{E}(\|S^{h}(t-s)\|^2_{\mathcal{L}(H)}\cdot\|A_{h}^{\beta}P^{h}G(X^{h}(s))u\|^2)
			\le C\mathbb{E}\|A_{h}^{\frac{1}{2}}P^{h}G(X^{h}(s))u\|^2\\
			\le&\ C\mathbb{E}\|A^{\frac{1}{2}}G(X^{h}(s))u\|^2 
			\le C\mathbb{E}(\|X^h(s)\|_1^2\cdot\|u\|^2)\le C\|u\|^2.
		\end{aligned}
	\end{equation*}
	By \eqref{eq2.2.5}, \eqref{eq2.3.2}, H\"older inequality, Lemma \ref{l2.3.1} and \eqref{eq2.3.12}, we have
	\begin{equation*}
		\begin{aligned}
			II \le&\ 3T\mathbb{E}\int_s^t\|A_{h}^{\beta}S^{h}(t-r)P^{h}\|^2_{\mathcal{L}(H)}\|DF(X^{h}(r)) \mathcal{D}_s^uX^{h}(r)\|^2dr\\
			\le&\ 3T\mathbb{E}\int_s^tC(t-r)^{-2\beta}(1+\|X^{h}(r)\|_V^2)^2\|\mathcal{D}_s^uX^{h}(r)\|^2dr\\
			\le&\ 3T\int_s^tC(t-r)^{-2\beta}(\mathbb{E}(1+\|X^{h}(r)\|_V^2)^4)^{\frac{1}{2}}(\mathbb{E}\|\mathcal{D}_s^uX^{h}(r)\|^4)^{\frac{1}{2}}dr\\
			\le&\ 3T\int_s^tC(t-r)^{-2\beta}\cdot C\cdot C\|u\|^2dr \le C\|u\|^2,
		\end{aligned}
	\end{equation*}
	here we use the fact that $\int_s^t(t-r)^{-2\beta}dr<\infty$.
	Similarly, we get
	\begin{equation*}
		III \le C\|u\|^2.
	\end{equation*}
	Thus we have proved that
	\begin{equation*}
		\mathbb{E}\|A_{h}^{\beta}\mathcal{D}_s^uX^{h}(t)\|^2 \le C\|u\|^2,
	\end{equation*}
	where $C=C(\xi,T,h_0,\beta)$ is a positive constant.
\end{proof}

\section{Weak convergence analysis for spatial semi-discretization}
\label{Sect5}

In this section, we turn to analyze the weak convergence order of the semi-discretized FEM approximation. For any given $\phi\in \mathcal{C}_b^2(H)$, define a weak error
\begin{equation*}
	\begin{aligned}
		e_w =&\ |\mathbb{E}[\phi(X(T))]-\mathbb{E}[\phi(X^{h}(T))]|\\
		\le&\ |\mathbb{E}[\phi(X(T))]-\mathbb{E}[\phi(X_{\tau}(T))]| + |\mathbb{E}[\phi(X_{\tau}(T))]-\mathbb{E}[\phi(X^{h}(T))]|.
	\end{aligned}
\end{equation*}
An estimate for the first term has been obtained in Theorem \ref{t2.4.1}. Now we focus on the second term.

\begin{theorem}\label{th2.5.1}
	Assume {\bf(A1)}-{\bf(A4)}. Then for any given $0<\varepsilon<1$ and $\phi \in \mathcal{C}_b^2(H)$, there exists a positive constant $C = C(\xi,T,h_0,\tau_0,\varepsilon,\phi)$, such that for all $h \in (0,h_0)$ and $\tau\in(0,\tau_0)$
	\begin{equation*}
		|\mathbb{E}[\phi(X_{\tau}(T))]-\mathbb{E}[\phi(X^{h}(T))]| \le C(\tau+ h^{1-\varepsilon}).
	\end{equation*}
\end{theorem}

\begin{proof}
	Obviously, for any given $\phi \in \mathcal{C}_b^2(H)$ and $\mathcal{F}_{T}$-measurable, $H$-valued random variable $\eta$, there holds \cite[(5.1)]{Andersson2016}
	\begin{equation*}
		\mathbb{E}[\phi(\eta)] = \mathbb{E}[\mathbb{E}[\phi(\eta)|\mathcal{F}_T]] = \mathbb{E}[\mathbb{E}[\phi(X_{\tau}(0,\eta))|\mathcal{F}_T]] = \mathbb{E}[U_{\tau}(0,\eta)].
	\end{equation*}
	Then we can split the weak error as
	\begin{equation*}
		\begin{aligned}
			&\ |\mathbb{E}[\phi(X_{\tau}(T))] - \mathbb{E}[\phi(X^{h}(T))]| = |\mathbb{E}[U_{\tau}(T,X(0))] - \mathbb{E}[U_{\tau}(0,X^{h}(T))]|\\
			\le&\ |\mathbb{E}[U_{\tau}(T,X(0))] - \mathbb{E}[U_{\tau}(T,X^{h}(0))]|
			+ |\mathbb{E}[U_{\tau}(T,X^{h}(0))] - \mathbb{E}[U_{\tau}(0,X^{h}(T))]|,
		\end{aligned}
	\end{equation*}
	where $X(0)=\xi$ and $X^{h}(0)=P^{h}\xi$.
	
	By the mean value theorem, \eqref{eq2.2.1.1}, \eqref{eq2.3.2} and Lemma \ref{l2.4.3}, we can bound the first term 
	\begin{equation*}
		\begin{aligned}
			&\ |\mathbb{E}[U_{\tau}(T,X(0))] - \mathbb{E}[U_{\tau}(T,X^{h}(0))]|\\
			=&\ \big|\mathbb{E}\int_0^1 DU_{\tau}(T,\theta \xi+(1-\theta)P^{h}\xi)(I-P^{h})\xi d\theta\big|\\
			=&\ \big|\mathbb{E}\int_0^1 \big( DU_{\tau}(T,\theta \xi+(1-\theta)P^{h}\xi),(I-P^{h})\xi\big) d\theta\big|\\
			\le&\ \mathbb{E}\int_0^1 \big|\big(DU_{\tau}(T,\theta \xi+(1-\theta)P^{h}\xi),(I-P^{h})A^{-\frac{1}{2}}A^{\frac{1}{2}}\xi\big)\big|d\theta\\
			\le&\ \mathbb{E}\int_0^1\|DU_{\tau}(T,\theta \xi+(1-\theta)P^{h}\xi)\|\cdot\|(I-P^{h})A^{-\frac{1}{2}}\|_{\mathcal{L}(H)}\|A^{\frac{1}{2}}\xi\|d\theta\\
			\le&\ C(T,\tau_0,\phi)\|(I-P^{h})A^{-\frac{1}{2}}\|_{\mathcal{L}(H)}\|\xi\|_1\\
			\le&\ C(\xi,T,h_0,\tau_0,\phi)h.
		\end{aligned}
	\end{equation*}
	
	Now we aim to estimate the second term $|\mathbb{E}[U_{\tau}(T,X^{h}(0)) - U_{\tau}(T,X^{h}(T))]|$.
	Applying the mild It\^o formula to $U_{\tau}(T-t,X^{h}(t))$  for $t \in [0,T]$, we get
	\begin{equation}\label{eq2.5.1}
		\begin{aligned}
			&\ U_{\tau}(0,X^{h}(T)) - U_{\tau}(T,X^{h}(0))\\
			=&\ \int_0^T\! -\frac{\partial U_{\tau}}{\partial t}(T-t,X^{h}(t)) + \big(DU_{\tau}(T-t,X^{h}(t)),-A_{h}X^{h}(t)+P^{h}F(X^{h}(t))\big)dt\\
			&+ \int_0^T\big(DU_{\tau}(T-t, X^{h}(t)),P^{h}G(X^{h}(t))dW(t)\big)\\
			&+ \frac{1}{2}\int_0^T\mathrm{Tr}\{D^2U_{\tau}(T-t,X^{h}(t))P^{h}G(X^{h}(t))(P^{h}G(X^{h}(t)))^*\}dt.
		\end{aligned}
	\end{equation}
	According to the Kolmogorov equation \eqref{eq2.4.13}, we have
	\begin{equation*}
		\begin{aligned}
			\frac{\partial U_{\tau}}{\partial t}(T-t,X^{h}(t)) =&\ (DU_{\tau}(T-t,X^{h}(t)), -AX^{h}(t)+F_{\tau}(X^{h}(t)))\\
			&\ + \frac{1}{2}\mathrm{Tr}\{D^2U_{\tau}(T-t,X^{h}(t))G_{\tau}(X^{h}(t))G_{\tau}^*(X^{h}(t))\}.
		\end{aligned}
	\end{equation*}
	Substitute this into \eqref{eq2.5.1}, we obtain
	\begin{equation*}
		\begin{aligned}
			&\ |\mathbb{E}[U_{\tau}(T, X^{h}(0))]-\mathbb{E}[\phi(X^{h}(T))]|\\ 
			\le&\ \big|\mathbb{E}\int_{0}^{T}((A-A_{h}) X^{h}(t), DU_{\tau}(T-t, X^{h}(t))) dt\big|\\
			&\ +\big|\mathbb{E}\int_{0}^{T}( F_{\tau}(X^{h}(t))-P^{h}F(X^{h}(t)), D U_{\tau}(T-t,X^{h}(t))) d t\big|\\
			&\ + \frac{1}{2}\big| \mathbb{E}\int_0^T \mathrm{Tr}\big\{D^2U_{\tau}(T-t, X^{h}(t))G_{\tau}(X^{h}(t))G_{\tau}^*(X^{h}(t))\\
			&\quad\quad\quad-D^2U_{\tau}(T-t, X^{h}(t))P^{h}G(X^{h}(t))(P^{h}G(X^{h}(t)))^*\big\} dt\big|\\
			:=&\ e^1 + e^2 + e^3.
		\end{aligned}
	\end{equation*}
	
	First we estimate $e^1$. In \cite{Thomee2006}, it is pointed out that $R^{h} = A_{h}^{-1}P^{h}A$ on $\dot{H}^1$ and $P^{h}A_{h} = A_{h}P^{h}$, $P^{h}A = AP^{h}$ on $V_h^0$. Then we have
	\begin{equation*}
		P^{h}A-A_{h}P^{h} = A_{h}(A_{h}^{-1}P^{h}A-P^{h})= A_{h}P^{h}(A_{h}^{-1}P^{h} A-I)=A_{h}P^{h}(R^{h}-I),\ \text{on }V_h^0
	\end{equation*}
	Thus we have
	\begin{equation*}
		\begin{aligned}
			&\ ((A-A_{h}) X^{h}(t), D U_{\tau}(T-t, X^{h}(t))) \\
			=&\ ((AP^{h}-P^{h} A_{h}) X^{h}(t), DU_{\tau}(T-t, X^{h}(t))) \\
			=&\ ( X^{h}(t),(P^{h} A-A_{h}P^{h}) DU_{\tau}(T-t, X^{h}(t))) \\
			=&\ ( X^{h}(t), A_{h}P^{h}(A_{h}^{-1}P^{h} A-I) D U_{\tau}(T-t, X^{h}(t))) \\
			=&\ ( X^{h}(t), A_{h}P^{h}(R^{h}-I) D U_{\tau}(T-t, X^{h}(t))).
		\end{aligned}
	\end{equation*}
	Substitute the mild solution $X^{h}(t)$ in \eqref{eq2.3.6} into above equation, we obtain
	\begin{equation*}
		\begin{aligned}
			e^1
			\le&\ \mathbb{E}\int_{0}^{T}\big|( S^{h}(t) X^{h}(0), A_{h}P^{h}(R^{h}-I) DU_{\tau}(T-t, X^{h}(t)))\big| dt+ \\
			&\ \mathbb{E}\int_{0}^{T}\big|\big(\!\int_{0}^{t} S^{h}(t-s)P^{h}F(X^{h}(s)) d s, A_{h}P^{h}(R^{h}-I) DU_{\tau}(T-t, X^{h}(t))\big)\big| dt+\\
			&\ \big|\mathbb{E}\int_{0}^{T}\big(\int_{0}^{t} S^{h}(t-s)P^{h} G(X^{h}(s)) dW(s), A_{h}P^{h}(R^{h}-I) DU_{\tau}(T-t, X^{h}(t))\big)dt\big|\\
			:=&\ e^{1,1}+e^{1,2}+e^{1,3}.
		\end{aligned}
	\end{equation*}
	
	Applying Cauchy-Schwartz inequality, \eqref{eq2.3.0}, \eqref{eq2.3.2} and Lemma \ref{l2.4.3}, we have for $0<\varepsilon<1$
	\begin{equation*}
		\begin{aligned}
			e^{1,1}
			=&\ \mathbb{E}\int_{0}^{T}|( A_{h}^{1-\frac{1}{2}\varepsilon}S^{h}(t) A_{h}^{\frac{1}{2}}X^{h}(0), A_{h}^{-\frac{1}{2}+\frac{1}{2}\varepsilon}P^{h}(R^{h}-I) DU_{\tau}(T-t, X^{h}(t)))| dt \\
			\le&\ \mathbb{E}\int_{0}^{T} \|A_{h}^{1-\frac{1}{2}\varepsilon}S^{h}(t)A_{h}^{\frac{1}{2}}P^{h}\xi\|\cdot \|A_{h}^{-\frac{1}{2}+\frac{1}{2}\varepsilon}P^{h}(R^{h}-I) DU_{\tau}(T-t, X^{h}(t))\| dt\\
			\le&\ \mathbb{E}\!\!\int_{0}^{T}\!\! \|A_{h}^{1-\frac{1}{2}\varepsilon}S^{h}(t)\|_{\mathcal{L}(H)}\|A_{h}^{\frac{1}{2}}P^{h}\xi\|\!\cdot\! \|A_{h}^{-\frac{1}{2}+\frac{1}{2}\varepsilon}P^{h}(R^{h}-I)\|_{\mathcal{L}(H)} \|DU_{\tau}(T-t, X^{h}(t))\| dt\\
			\le&\ \int_{0}^{T} C(h_0,\varepsilon)t^{-1+\frac{1}{2}\varepsilon}\cdot\|\xi\|_1\cdot C(h_0,\varepsilon)h^{1-\varepsilon}\cdot\mathbb{E}\|DU_{\tau}(T-t, X^{h}(t))\| dt\\
			\le&\ \int_{0}^{T} C(h_0,\varepsilon)t^{-1+\frac{1}{2}\varepsilon}\cdot\|\xi\|_1\cdot C(h_0,\varepsilon)h^{1-\varepsilon}\cdot C(T,\tau_0,\phi)dt\\
			\le&\ C(\xi,T,h_0,\tau_0,\varepsilon,\phi)h^{1-\varepsilon}.
		\end{aligned}
	\end{equation*}
	Here we have used the fact that $\int_0^Tt^{-1+\frac{1}{2}\varepsilon}dt=C(T,\varepsilon)<+\infty$.
	Applying \eqref{eq2.2.5} and Lemma \ref{l2.3.1}, we get
	\begin{equation}\label{eq2.5.3}
		\mathbb{E}\|F(X^{h}(t))\|_1 \le C_F\mathbb{E}\big((1+\|X^{h}(t)\|_V^2)\|X^{h}(t)\|_1\big) \le C(\xi,T).
	\end{equation}
	Then, in a similar way, we get
	\begin{equation*}
		e^{1,2} \le C(\xi,T,h_0,\tau_0,\varepsilon,\phi)h^{1-\varepsilon}.
	\end{equation*}
	According to the Malliavin integration by parts formula, Lemmas \ref{l2.2.1} and \ref{l2.4.3}, Cauchy-Schwartz inequality, \eqref{eq2.2.0}, \eqref{eq2.2.1}, \eqref{eq2.2.00}, \eqref{eq2.3.2}, H\"older inequality and Theorem \ref{th2.3.1}, we obtain
	\begin{equation*}
		\begin{aligned}
			e^{1,3}
			=&\ \big|\mathbb{E}\int_{0}^{T}\big(\int_{0}^{t} S^{h}(t-s)P^{h} G(X^{h}(s)) dW(s),A_{h}P^{h}(R^{h}-I) DU_{\tau}(T-t, X^{h}(t))\big) dt\big|\\
			=&\ \big|\mathbb{E}\int_{0}^{T}\int_{0}^{t}\big( S^{h}(t-s)P^{h} G(X^{h}(s)),A_{h}P^{h}(R^{h}-I)D^2U_{\tau}(T-t, X^{h}(t))\mathcal{D}_sX^{h}(t)\big)_{\mathcal{L}_2(H)} dsdt\big|\\
			=&\ \big|\mathbb{E}\int_{0}^{T}\int_{0}^{t}\big( A_{h}^{1-\frac{1}{2}\varepsilon}S^{h}(t-s)A_{h}^{\frac{1}{2}}P^{h}G(X^{h}(s)), A_{h}^{-\frac{1}{2}+\frac{1}{2}\varepsilon}P^{h}(R^{h}-I)\cdot\\
			&\quad\quad\quad\quad\quad D^2U_{\tau}(T-t, X^{h}(t))A_{h}^{-\frac{1}{2}+\frac{1}{2}\varepsilon}A_{h}^{\frac{1}{2}-\frac{1}{2}\varepsilon}\mathcal{D}_sX^{h}(t)\big)_{\mathcal{L}_2(H)} dsdt\big|\\
			\le&\ \mathbb{E}\int_{0}^{T}\!\int_{0}^{t}\|A_{h}^{1-\frac{1}{2}\varepsilon}S^{h}(t-s)\|_{\mathcal{L}(H)}\|A_{h}^{\frac{1}{2}}P^{h}G(X^{h}(s))\|_{\mathcal{L}_2(H)}\|A_{h}^{-\frac{1}{2}+\frac{1}{2}\varepsilon}P^{h}(R^{h}-I)\|_{\mathcal{L}(H)}\cdot\\
			&\quad\quad\quad\quad\quad\|D^2U_{\tau}(T-t, X^{h}(t))\|_{\mathcal{L}(H)}\|A_{h}^{-\frac{1}{2}+\frac{1}{2}\varepsilon}\|_{\mathcal{L}_2(H)}\|A_{h}^{\frac{1}{2}-\frac{1}{2}\varepsilon}\mathcal{D}_sX^{h}(t)\|_{\mathcal{L}(H)} dsdt\\
			\le&\ \mathbb{E}\int_{0}^{T}\int_{0}^{t} C(h_0,\varepsilon)(t-s)^{-1+\frac{1}{2}\varepsilon}\cdot C\|X^{h}(s)\|_1\cdot C(h_0,\varepsilon)h^{1-\varepsilon}\cdot\\
			&\quad\quad\quad\quad\quad C(T,\tau_0,\phi)(1+\|X^{h}(t)\|_1^3)\cdot C\cdot\|A_{h}^{\frac{1}{2}-\frac{1}{2}\varepsilon}\mathcal{D}_sX^{h}(t)\|_{\mathcal{L}(H)} dsdt\\
			\le&\ C(T,h_0,\tau_0,\varepsilon,\phi)h^{1-\varepsilon}\int_{0}^{T}\int_{0}^{t} (t-s)^{-1+\frac{1}{2}\varepsilon}\big(\mathbb{E}\|X^{h}(s)\|_1^4\big)^{\frac{1}{4}}\cdot\\
			&\quad\quad\quad\quad\quad\big(\mathbb{E}(1+\|X^{h}(t)\|_1^3)^4\big)^{\frac{1}{4}}(\mathbb{E}\|A_{h}^{\frac{1}{2}-\frac{1}{2}\varepsilon}\mathcal{D}_sX^{h}(t)\|^2_{\mathcal{L}(H)})^{\frac{1}{2}}dsdt\\
			\le&\ C(\xi,T,h_0,\tau_0,\varepsilon,\phi)h^{1-\varepsilon}.
		\end{aligned}
	\end{equation*}
	Here we use the fact that $ \displaystyle\int_{0}^{T}\int_{0}^{t}(t-s)^{-1+\frac{1}{2}\varepsilon}dsdt = (\frac{1}{2}\varepsilon(\frac{1}{2}\varepsilon+1))^{-1}T^{\frac{1}{2}\varepsilon+1}<+\infty$.
	Thus, we conclude that 
	\begin{equation*}
		e^1 \le Ch^{1-\varepsilon}.
	\end{equation*}
	
	Next, we turn to investigate $e^2$. By triangle inequality, we have
	\begin{equation*}
		\begin{aligned}
			e^2 
			=&\ \big|\mathbb{E}\int_{0}^{T}\big( F_{\tau}(X^{h}(t))\pm F(X^{h}(t))-P^{h}F(X^{h}(t)), DU_{\tau}(T-t,X^{h}(t))\big) dt\big|\\
			\le&\ 
			\mathbb{E}\int_{0}^{T}\big|\big( F_{\tau}(X^{h}(t))-F(X^{h}(t)), D U_{\tau}(T-t,X^{h}(t))\big)\big| dt\\
			&+ \mathbb{E}\int_{0}^{T}\big|\big( (I-P^{h})F(X^{h}(t)), D U_{\tau}(T-t,X^{h}(t))\big)\big| dt\\
			:=&\ e^{2,1}+e^{2,2}.
		\end{aligned}
	\end{equation*}
	By Cauchy-Schwartz inequality, \eqref{eq2.4.9}, Lemmas \ref{l2.4.3} and \ref{l2.3.1}, we have
	\begin{equation*}
		\begin{aligned}
			e^{2,1} \le&\ \mathbb{E}\int_{0}^{T}\|F_{\tau}(X^{h}(t))-F(X^{h}(t))\|\cdot \|D U_{\tau}(T-t,X^{h}(t))\| dt\\
			\le&\ \int_{0}^{T}\mathbb{E}(1+\|X^{h}(t)\|_V^5)\tau \cdot C(T,\tau_0,\phi)dt\\
			\le&\ C(\xi,T,\tau_0,\phi)\tau.
		\end{aligned}
	\end{equation*}
	By Cauchy-Schwartz inequality, \eqref{eq2.2.5}, \eqref{eq2.3.2}, \eqref{eq2.5.3}, Lemmas \ref{l2.4.3} and \ref{l2.3.1}, we have
	\begin{equation*}
		\begin{aligned}
			e^{2,2} =&\ \mathbb{E}\int_{0}^{T}\big|\big( (I-P^{h})A^{-\frac{1}{2}}A^{\frac{1}{2}}F(X^{h}(t)), DU_{\tau}(T-t,X^{h}(t))\big)\big| dt\\
			\le&\ \mathbb{E}\int_0^T \|(I-P^{h})A^{-\frac{1}{2}}\|_{\mathcal{L}(H)}\cdot\|A^{\frac{1}{2}}F(X^{h}(t))\|\cdot \|DU_{\tau}(T-t,X^{h}(t))\| dt\\
			\le&\ \int_0^TC(h_0,\varepsilon)h\cdot\mathbb{E}\|F(X^{h}(t))\|_1\cdot C(T,\tau_0,\phi)dt\\
			\le&\ C(\xi,T,h_0,\tau_0,\phi)h.
		\end{aligned}
	\end{equation*}
	Thus, we conclude that 
	\begin{equation*}
		e^2 \le C(\tau+h).
	\end{equation*}
	
	At last, we study $e^3$. Obviously, we have
	\begin{equation*}
		\begin{aligned}
			&\ G_{\tau}(\eta)G_{\tau}^*(\eta)-P^{h}G(\eta)G^*(\eta)P^{h}\\
			=&\ G_{\tau}(\eta)G_{\tau}^*(\eta)\pm G_{\tau}(\eta)G^*(\eta) \pm G(\eta)G^*(\eta) \pm P^{h}G(\eta)G^*(\eta)-P^{h}G(\eta)G^*(\eta)P^{h}\\
			=&\ G_{\tau}(\eta)(G_{\tau}^*(\eta)-G^*(\eta)) +  (G_{\tau}(\eta)-G(\eta))G^*(\eta)\\
			&\ + (I-P^{h})G(\eta)G^*(\eta)+P^{h}G(\eta)G^*(\eta)(I-P^{h}).
		\end{aligned}
	\end{equation*}
	Then
	\begin{equation*}
		\begin{aligned}
			e^3 \le&\ \frac{1}{2} \mathbb{E}\int_0^T \big|\mathrm{Tr}\{G_{\tau}(X^{h}(t))\big(G_{\tau}^*(X^{h}(t))-G^*(X^{h}(t))\big)D^2U_{\tau}(T-t, X^{h}(t))\}\big|dt\\
			&\ +\frac{1}{2} \mathbb{E}\int_0^T \big|\mathrm{Tr}\{\big(G_{\tau}(X^{h}(t))-G(X^{h}(t))\big)G^*(X^{h}(t))D^2U_{\tau}(T-t, X^{h}(t))\}\big|dt\\
			&\ +\frac{1}{2}\mathbb{E}\int_0^T |\mathrm{Tr}\{ P^{h}G(X^{h}(t))G^*(X^{h}(t))(I-P^{h})D^2U_{\tau}(T-t, X^{h}(t))\}|dt \\
			&\ +\frac{1}{2}\mathbb{E}\int_0^T| \mathrm{Tr}\{ (I-P^{h})G(X^{h}(t))G^*(X^{h}(t))D^2U_{\tau}(T-t, X^{h}(t))\}| dt\\
			:=&\ e^{3,1} + e^{3,2} + e^{3,3}+e^{3,4}. 
		\end{aligned}
	\end{equation*}
	
	By \eqref{eq2.2.0}, \eqref{eq2.2.1}, Lemmas \ref{l2.4.2}, \ref{l2.4.3},  and \ref{l2.3.1} we have
	\begin{equation*}
		\begin{aligned}
			e^{3,1}
			\le&\ \frac{1}{2}\mathbb{E}\int_0^T \|G_{\tau}(X^{h}(t))\big(G_{\tau}^*(X^{h}(t))-G^*(X^{h}(t))\big)\|_{\mathcal{L}_1(H)}\cdot\|D^2U_{\tau}(T-t, X^{h}(t))\|_{\mathcal{L}(H)}dt\\
			\le&\ \frac{1}{2}\mathbb{E}\int_0^T \|G_{\tau}(X^{h}(t))\|_{\mathcal{L}(H)}\|G_{\tau}(X^{h}(t))-G(X^{h}(t))\|_{\mathcal{L}_2(H)}\cdot C(T,\tau_0,\phi)(1+\|X^{h}(t)\|_1^3)dt\\
			\le&\ C(T,\tau_0,\phi)\mathbb{E}\int_0^TC(1+\|X^{h}(t)\|_V)\cdot C\tau(1+\|X^{h}(t)\|_V^3)\cdot (1+\|X^{h}(t)\|_1^3) dt\\
			\le&\ C(\xi,T,\tau_0,\phi)\tau.
		\end{aligned}
	\end{equation*}
	Similarly, we obtain $e^{3,2}\le C(\xi,T,\tau_0,\phi)\tau$.
	By \eqref{eq2.2.0}, \eqref{eq2.2.1}, \eqref{eq2.3.2}, Lemmas \ref{l2.2.1}, \ref{l2.4.3}, \ref{l2.3.1}, we have
	\begin{equation*}
		\begin{aligned}
			e^{3,3}
			=&\ |\frac{1}{2}\mathbb{E}\int_0^T \mathrm{Tr}\{P^{h}G(X^{h}(t))G^*(X^{h}(t))A^{\frac{1}{2}}(I-P^{h})A^{-\frac{1}{2}}D^2U_{\tau}(T-t, X^{h}(t))\}dt|\\
			\le&\ \frac{1}{2}\mathbb{E}\int_0^T  \|P^{h}G(X^{h}(t))(A^{\frac{1}{2}}G(X^{h}(t)))^*\|_{\mathcal{L}_2(H)}\cdot\|(I-P^{h})A^{-\frac{1}{2}}D^2U_{\tau}(T-t, X^{h}(t))\|_{\mathcal{L}(H)}dt \\
			\le&\ \frac{1}{2}\mathbb{E}\int_0^T  \|P^{h}G(X^{h}(t))\|_{\mathcal{L}(H)}\cdot\|A^{\frac{1}{2}}G(X^{h}(t))\|_{\mathcal{L}_2(H)}\cdot\|(I-P^{h})A^{-\frac{1}{2}}\|_{\mathcal{L}(H)}\cdot\\
			&\quad\quad\quad\quad\|D^2U_{\tau}(T-t, X^{h}(t))\|_{\mathcal{L}(H)}dt \\
			\le&\ \frac{1}{2}\mathbb{E}\int_0^T  C(1+\|X^{h}(t)\|)\cdot C\|X^h(t)\|_1\cdot C(T,\tau_0,\phi)(1+\|X^{h}(t)\|_1^3)dt\\
			\le&\ C(T,h_0,\tau_0,\phi)h\cdot\mathbb{E}\int_0^T  (1+\|X^{h}(t)\|)\cdot\|X^{h}(t)\|_1\cdot(1+\|X^{h}(t)\|_1^3)dt \\
			\le&\ C(\xi,T,h_0,\tau_0,\phi)h.
		\end{aligned}
	\end{equation*}	
	Similarly, we can prove $e^{3,4} \le Ch$. 
	Thus, we have proved
	\begin{equation*}
		e^{3} \le C(\xi,T,h_0,\tau_0,\phi)(\tau+h).
	\end{equation*} 
	
	In terms of the estimates for $e^1$, $e^2$ and $e^3$, we finally obtain
	\begin{equation*}
		|\mathbb{E}[\phi(X_{\tau}(T))]-\mathbb{E}[\phi(X^{h}(T))]| \le C(\xi,T,h_0,\tau_0,\varepsilon,\phi)(\tau+ h^{1-\varepsilon}).
	\end{equation*}
	Thus, the proof is completed.
\end{proof}

Now, we are ready to study the weak convergence order for the semi-discretized FEM approximation.

\begin{theorem}\label{th2.5.2}
	Assume {\bf(A1)}-{\bf(A4)}. Then for any given $0<\varepsilon<1$ and $\phi \in \mathcal{C}_b^2(H)$, there exists a positive constant $C = C(\xi,T,h_0,\varepsilon,\phi)$, such that for all $h \in (0,h_0)$
	\begin{equation*}
		|\mathbb{E}[\phi(X(T))]-\mathbb{E}[\phi(X^{h}(T))]| \le Ch^{1-\varepsilon}.
	\end{equation*}
\end{theorem}

\begin{proof}
	According to Theorems \ref{t2.4.1} and \ref{th2.5.1}, we obtain
	\begin{equation*}
		\begin{aligned}
			&\ |\mathbb{E}[\phi(X(T))]-\mathbb{E}[\phi(X^{h}(T))]|\\ 
			\le&\ |\mathbb{E}[\phi(X(T))]-\mathbb{E}[\phi(X_{\tau}(T))]|+ |\mathbb{E}[\phi(X_{\tau}(T))]-\mathbb{E}[\phi(X^{h}(T))]|\\
			\le&\ C(\mathbb{E}\|X(T)-X_{\tau}(T)\|^2)^{\frac{1}{2}}+ |\mathbb{E}[\phi(X_{\tau}(T))]-\mathbb{E}[\phi(X^{h}(T))]|\\
			\le&\ C\tau+C(\tau+ h^{1-\varepsilon}).
		\end{aligned}
	\end{equation*}
	Let $\tau\to 0$ in above equation, then we finish the proof of theorem.
\end{proof}

\section{Numerical experiments}
\label{Sect6}
Consider a stochastic Allen-Cahn equation for $(t,x)\in (0,1)\times(0,1)$
\begin{equation*}
	dX(t,x)-\Delta X(t,x)dt = (X(t,x)-X^3(t,x))dt+G(X(t,x))dW(t),
\end{equation*}
where $G(X) = I+B_1X+\sin(X)A^{-\frac{1}{4}-\frac{1}{2}\delta}$ with $\delta =0.001$, $B_1$ is a bounded linear operator on $\dot{H}^{-\frac{1}{2}-\delta}$ determined by $B_1\varphi_j = j^{-1}\varphi_j$. Therefore, $G$ satisfies the assumption {\bf(A2)}. We take an initial value function $\xi \in \dot{H}^1$ of the form
\begin{equation*}
	\xi(x)=\begin{cases}
		x,x\in[0,\frac{1}{2}),\\
		1-x,x\in[\frac{1}{2},1],
	\end{cases}
\end{equation*}
and choose $\phi(X) = \sin (\|X\|)$ to measure the weak error. Fix time step size $\kappa = \frac{1}{1000}$, we compute $M=20000$ independent sample trajectories for each  spatial step size $h_i = 2^{-i}$ $(i=5,6,7,8,12)$, respectively, and regard the numerical solution computed with step size $h_{12} = 2^{-12}$ as the reference 'real' solution. In order to emphasize the feature of weak convergence, we simultaneously calculate strong and weak errors and corresponding convergence orders, which are shown in Table \ref{table2.1}. 
We also plot the errors versus $h$ in log-log scale to observe the convergence order, see Figure \ref{figure2.1}.

\begin{table*}[h]
	\begin{center}  
	\centering
		\caption{Errors and convergence orders}
		\label{table2.1}
		\begin{tabular}{c|cc|cc} 
			\toprule     
			$h$ & strong error & order & weak error & order\\ 
			\midrule   
			$2^{-5}$ & 7.3009 $\times 10^{-3}$ & &  3.0341 $\times 10^{-3}$ & \\
			$2^{-6}$ & 4.6211 $\times 10^{-3}$ & 0.6184 & 1.3887 $\times 10^{-3}$ & 1.1275\\
			$2^{-7}$ & 2.9596 $\times 10^{-3}$ & 0.6428 & 6.3082 $\times 10^{-4}$ & 1.1385\\
			$2^{-8}$ & 1.9612 $\times 10^{-3}$ &  0.5937 & 2.9261 $\times 10^{-4}$ & 1.1082\\
			\bottomrule     
		\end{tabular}
	\end{center}   
\end{table*}

\begin{figure}[H]
	\centering
	\includegraphics[width=0.5\linewidth]{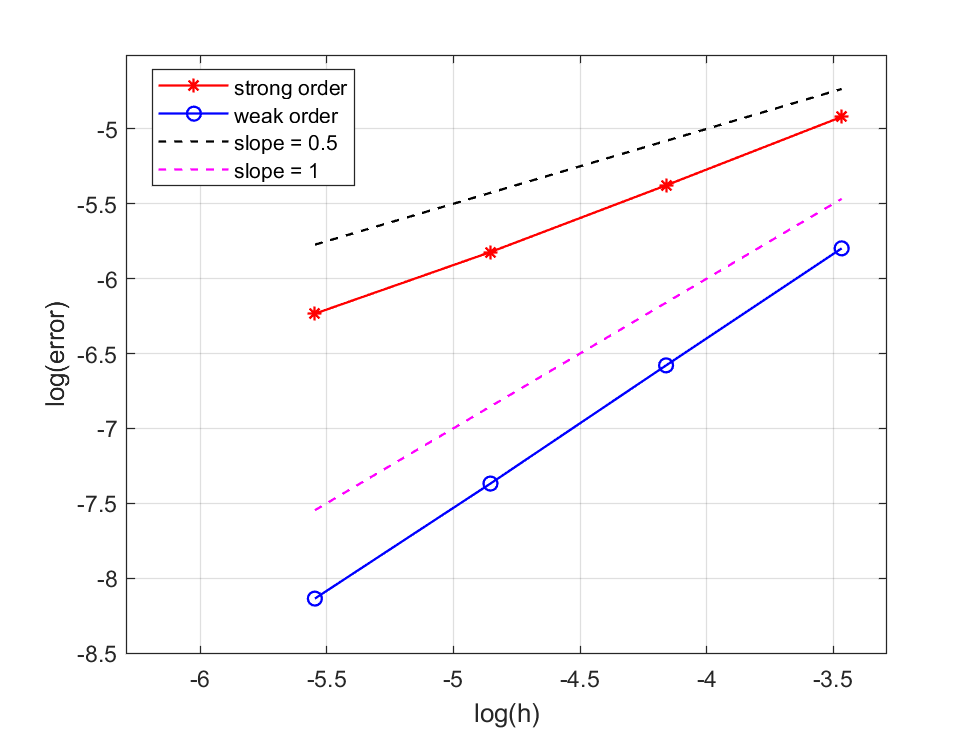}
	\caption{Convergence order in $\log-\log$ scale}
	\label{figure2.1}
\end{figure}

\section{Conclusion and Discussion}
In this paper, we investigated the optimal weak convergence order for an FEM approximation to a stochastic Allen-Cahn equation driven by multiplicative white noise. We first constructed an auxiliary equation based on the splitting-up technique and derived the strong convergence order of 1 in time between the auxiliary and exact solutions. Then, we analyzed the weak errors between the FEM approximation and the exact solution and obtained the optimal weak convergence order in space. 

It is well known that high-dimensional SPDEs driven by space-time white noise are generally ill-posed. Consequently, the results presented in this paper do not extend to the high-dimensional stochastic Allen-Cahn equation driven by multiplicative noise. We plan to address this challenge in future work by applying the methodologies developed in this paper.

\section*{Acknowledgments}
This work is supported by Jilin Provincial Department of Science and Technology grant 20240301017GX; National Natural Science Foundation of China grants 12171199, 11971198 and 22341302; National Key Research and Development Program of China grants 2020YFA0713602 and 2023YFA1008803; and Key Laboratory of Symbolic Computation and Knowledge Engineering of Ministry of Education of China housed at Jilin University.

\section*{Compliance with Ethical Standards}
The authors declare that there are no conflicts of interest. 

\bibliography{references}


\begin{thebibliography}{46}
\ifx \bisbn   \undefined \def \bisbn  #1{ISBN #1}\fi
\ifx \binits  \undefined \def \binits#1{#1}\fi
\ifx \bauthor  \undefined \def \bauthor#1{#1}\fi
\ifx \batitle  \undefined \def \batitle#1{#1}\fi
\ifx \bjtitle  \undefined \def \bjtitle#1{#1}\fi
\ifx \bvolume  \undefined \def \bvolume#1{\textbf{#1}}\fi
\ifx \byear  \undefined \def \byear#1{#1}\fi
\ifx \bissue  \undefined \def \bissue#1{#1}\fi
\ifx \bfpage  \undefined \def \bfpage#1{#1}\fi
\ifx \blpage  \undefined \def \blpage #1{#1}\fi
\ifx \burl  \undefined \def \burl#1{\textsf{#1}}\fi
\ifx \doiurl  \undefined \def \doiurl#1{\url{https://doi.org/#1}}\fi
\ifx \betal  \undefined \def \betal{\textit{et al.}}\fi
\ifx \binstitute  \undefined \def \binstitute#1{#1}\fi
\ifx \binstitutionaled  \undefined \def \binstitutionaled#1{#1}\fi
\ifx \bctitle  \undefined \def \bctitle#1{#1}\fi
\ifx \beditor  \undefined \def \beditor#1{#1}\fi
\ifx \bpublisher  \undefined \def \bpublisher#1{#1}\fi
\ifx \bbtitle  \undefined \def \bbtitle#1{#1}\fi
\ifx \bedition  \undefined \def \bedition#1{#1}\fi
\ifx \bseriesno  \undefined \def \bseriesno#1{#1}\fi
\ifx \blocation  \undefined \def \blocation#1{#1}\fi
\ifx \bsertitle  \undefined \def \bsertitle#1{#1}\fi
\ifx \bsnm \undefined \def \bsnm#1{#1}\fi
\ifx \bsuffix \undefined \def \bsuffix#1{#1}\fi
\ifx \bparticle \undefined \def \bparticle#1{#1}\fi
\ifx \barticle \undefined \def \barticle#1{#1}\fi
\bibcommenthead
\ifx \bconfdate \undefined \def \bconfdate #1{#1}\fi
\ifx \botherref \undefined \def \botherref #1{#1}\fi
\ifx \url \undefined \def \url#1{\textsf{#1}}\fi
\ifx \bchapter \undefined \def \bchapter#1{#1}\fi
\ifx \bbook \undefined \def \bbook#1{#1}\fi
\ifx \bcomment \undefined \def \bcomment#1{#1}\fi
\ifx \oauthor \undefined \def \oauthor#1{#1}\fi
\ifx \citeauthoryear \undefined \def \citeauthoryear#1{#1}\fi
\ifx \endbibitem  \undefined \def \endbibitem {}\fi
\ifx \bconflocation  \undefined \def \bconflocation#1{#1}\fi
\ifx \arxivurl  \undefined \def \arxivurl#1{\textsf{#1}}\fi
\csname PreBibitemsHook\endcsname

\bibitem[\protect\citeauthoryear{Cai et~al.}{2024}]{CaiMeng2024}
\begin{barticle}
\bauthor{\bsnm{Cai}, \binits{M.}},
\bauthor{\bsnm{Gan}, \binits{S.}},
\bauthor{\bsnm{Wang}, \binits{X.}}:
\batitle{Weak approximations of stochastic partial differential equations with
  fractional noise}.
\bjtitle{J. Comput. Math.}
\bvolume{42}(\bissue{3}),
\bfpage{735}--\blpage{754}
(\byear{2024})
\doiurl{10.4208/jcm.2203-m2021-0194}
\end{barticle}
\endbibitem

\bibitem[\protect\citeauthoryear{Cao et~al.}{2020}]{CaoYanzhao2020}
\begin{barticle}
\bauthor{\bsnm{Cao}, \binits{Y.}},
\bauthor{\bsnm{Hong}, \binits{J.}},
\bauthor{\bsnm{Liu}, \binits{Z.}}:
\batitle{Well-posedness and finite element approximations for elliptic {SPDE}s
  with {G}aussian noises}.
\bjtitle{Commun. Math. Res.}
\bvolume{36}(\bissue{2}),
\bfpage{113}--\blpage{127}
(\byear{2020})
\doiurl{10.4208/cmr.2020-0006}
\end{barticle}
\endbibitem

\bibitem[\protect\citeauthoryear{Chai et~al.}{2018}]{ChaiShimin2018}
\begin{barticle}
\bauthor{\bsnm{Chai}, \binits{S.}},
\bauthor{\bsnm{Cao}, \binits{Y.}},
\bauthor{\bsnm{Zou}, \binits{Y.}},
\bauthor{\bsnm{Zhao}, \binits{W.}}:
\batitle{Conforming finite element methods for the stochastic
  {C}ahn-{H}illiard-{C}ook equation}.
\bjtitle{Appl. Numer. Math.}
\bvolume{124},
\bfpage{44}--\blpage{56}
(\byear{2018})
\doiurl{10.1016/j.apnum.2017.09.010}
\end{barticle}
\endbibitem

\bibitem[\protect\citeauthoryear{Jentzen and Kloeden}{2009}]{Jentzen2009}
\begin{barticle}
\bauthor{\bsnm{Jentzen}, \binits{A.}},
\bauthor{\bsnm{Kloeden}, \binits{P.E.}}:
\batitle{The numerical approximation of stochastic partial differential
  equations}.
\bjtitle{Milan J. Math.}
\bvolume{77},
\bfpage{205}--\blpage{244}
(\byear{2009})
\doiurl{10.1007/s00032-009-0100-0}
\end{barticle}
\endbibitem

\bibitem[\protect\citeauthoryear{Lord and Tambue}{2013}]{Lord2013}
\begin{barticle}
\bauthor{\bsnm{Lord}, \binits{G.J.}},
\bauthor{\bsnm{Tambue}, \binits{A.}}:
\batitle{Stochastic exponential integrators for the finite element
  discretization of {SPDE}s for multiplicative and additive noise}.
\bjtitle{IMA J. Numer. Anal.}
\bvolume{33}(\bissue{2}),
\bfpage{515}--\blpage{543}
(\byear{2013})
\doiurl{10.1093/imanum/drr059}
\end{barticle}
\endbibitem

\bibitem[\protect\citeauthoryear{Zhang et~al.}{2022}]{ZhangFengshan2022}
\begin{barticle}
\bauthor{\bsnm{Zhang}, \binits{F.}},
\bauthor{\bsnm{Zou}, \binits{Y.}},
\bauthor{\bsnm{Chai}, \binits{S.}},
\bauthor{\bsnm{Zhang}, \binits{R.}},
\bauthor{\bsnm{Cao}, \binits{Y.}}:
\batitle{Splitting-up spectral method for nonlinear filtering problems with
  correlation noises}.
\bjtitle{J. Sci. Comput.}
\bvolume{93}(\bissue{1}),
\bfpage{25}--\blpage{24}
(\byear{2022})
\doiurl{10.1007/s10915-022-01994-6}
\end{barticle}
\endbibitem

\bibitem[\protect\citeauthoryear{Zhang et~al.}{2024}]{ZhangFengshan2024}
\begin{barticle}
\bauthor{\bsnm{Zhang}, \binits{F.}},
\bauthor{\bsnm{Zou}, \binits{Y.}},
\bauthor{\bsnm{Chai}, \binits{S.}},
\bauthor{\bsnm{Cao}, \binits{Y.}}:
\batitle{Numerical analysis of a time discretized method for nonlinear
  filtering problem with {L}\'evy process observations}.
\bjtitle{Adv. Comput. Math.}
\bvolume{50}(\bissue{4}),
\bfpage{73}--\blpage{32}
(\byear{2024})
\doiurl{10.1007/s10444-024-10169-w}
\end{barticle}
\endbibitem

\bibitem[\protect\citeauthoryear{Andersson and Larsson}{2016}]{Andersson2016}
\begin{barticle}
\bauthor{\bsnm{Andersson}, \binits{A.}},
\bauthor{\bsnm{Larsson}, \binits{S.}}:
\batitle{Weak convergence for a spatial approximation of the nonlinear
  stochastic heat equation}.
\bjtitle{Math. Comp.}
\bvolume{85}(\bissue{299}),
\bfpage{1335}--\blpage{1358}
(\byear{2016})
\doiurl{10.1090/mcom/3016}
\end{barticle}
\endbibitem

\bibitem[\protect\citeauthoryear{Debussche}{2011}]{Debussche2011}
\begin{barticle}
\bauthor{\bsnm{Debussche}, \binits{A.}}:
\batitle{Weak approximation of stochastic partial differential equations: the
  nonlinear case}.
\bjtitle{Math. Comp.}
\bvolume{80}(\bissue{273}),
\bfpage{89}--\blpage{117}
(\byear{2011})
\doiurl{10.1090/S0025-5718-2010-02395-6}
\end{barticle}
\endbibitem

\bibitem[\protect\citeauthoryear{Kov\'{a}cs et~al.}{2012}]{Kovacs2012}
\begin{barticle}
\bauthor{\bsnm{Kov\'{a}cs}, \binits{M.}},
\bauthor{\bsnm{Larsson}, \binits{S.}},
\bauthor{\bsnm{Lindgren}, \binits{F.}}:
\batitle{Weak convergence of finite element approximations of linear stochastic
  evolution equations with additive noise}.
\bjtitle{BIT}
\bvolume{52}(\bissue{1}),
\bfpage{85}--\blpage{108}
(\byear{2012})
\doiurl{10.1007/s10543-011-0344-2}
\end{barticle}
\endbibitem

\bibitem[\protect\citeauthoryear{Kruse}{2014}]{Kruse2014}
\begin{bbook}
\bauthor{\bsnm{Kruse}, \binits{R.}}:
\bbtitle{Strong and Weak Approximation of Semilinear Stochastic Evolution
  Equations}.
\bsertitle{Lecture Notes in Mathematics},
vol. \bseriesno{2093},
p. \bfpage{177}
(\byear{2014}).
\doiurl{10.1007/978-3-319-02231-4} .
\burl{https://doi.org/10.1007/978-3-319-02231-4}
\end{bbook}
\endbibitem

\bibitem[\protect\citeauthoryear{Yan}{2005}]{YanYubin2005}
\begin{barticle}
\bauthor{\bsnm{Yan}, \binits{Y.}}:
\batitle{Galerkin finite element methods for stochastic parabolic partial
  differential equations}.
\bjtitle{SIAM J. Numer. Anal.}
\bvolume{43}(\bissue{4}),
\bfpage{1363}--\blpage{1384}
(\byear{2005})
\doiurl{10.1137/040605278}
\end{barticle}
\endbibitem

\bibitem[\protect\citeauthoryear{Furihata et~al.}{2018}]{Kovacs2018_2}
\begin{barticle}
\bauthor{\bsnm{Furihata}, \binits{D.}},
\bauthor{\bsnm{Kov\'acs}, \binits{M.}},
\bauthor{\bsnm{Larsson}, \binits{S.}},
\bauthor{\bsnm{Lindgren}, \binits{F.}}:
\batitle{Strong convergence of a fully discrete finite element approximation of
  the stochastic {C}ahn-{H}illiard equation}.
\bjtitle{SIAM J. Numer. Anal.}
\bvolume{56}(\bissue{2}),
\bfpage{708}--\blpage{731}
(\byear{2018})
\doiurl{10.1137/17M1121627}
\end{barticle}
\endbibitem

\bibitem[\protect\citeauthoryear{Hutzenthaler and Jentzen}{2020}]{Jentzen2020}
\begin{barticle}
\bauthor{\bsnm{Hutzenthaler}, \binits{M.}},
\bauthor{\bsnm{Jentzen}, \binits{A.}}:
\batitle{On a perturbation theory and on strong convergence rates for
  stochastic ordinary and partial differential equations with nonglobally
  monotone coefficients}.
\bjtitle{Ann. Probab.}
\bvolume{48}(\bissue{1}),
\bfpage{53}--\blpage{93}
(\byear{2020})
\doiurl{10.1214/19-AOP1345}
\end{barticle}
\endbibitem

\bibitem[\protect\citeauthoryear{Jentzen et~al.}{2020}]{Jentzen2020_2}
\begin{barticle}
\bauthor{\bsnm{Jentzen}, \binits{A.}},
\bauthor{\bsnm{Lindner}, \binits{F.}},
\bauthor{\bsnm{Puvsnik}, \binits{P.}}:
\batitle{Exponential moment bounds and strong convergence rates for
  tamed-truncated numerical approximations of stochastic convolutions}.
\bjtitle{Numer. Algorithms}
\bvolume{85}(\bissue{4}),
\bfpage{1447}--\blpage{1473}
(\byear{2020})
\doiurl{10.1007/s11075-019-00871-y}
\end{barticle}
\endbibitem

\bibitem[\protect\citeauthoryear{Qi et~al.}{2023}]{XuChuanju2023}
\begin{barticle}
\bauthor{\bsnm{Qi}, \binits{X.}},
\bauthor{\bsnm{Zhang}, \binits{Y.}},
\bauthor{\bsnm{Xu}, \binits{C.}}:
\batitle{An efficient approximation to the stochastic {A}llen-{C}ahn equation
  with random diffusion coefficient field and multiplicative noise}.
\bjtitle{Adv. Comput. Math.}
\bvolume{49}(\bissue{5}),
\bfpage{73}--\blpage{24}
(\byear{2023})
\doiurl{10.1007/s10444-023-10072-w}
\end{barticle}
\endbibitem

\bibitem[\protect\citeauthoryear{Cerrai}{2003}]{Cerrai2003}
\begin{barticle}
\bauthor{\bsnm{Cerrai}, \binits{S.}}:
\batitle{Stochastic reaction-diffusion systems with multiplicative noise and
  non-{L}ipschitz reaction term}.
\bjtitle{Probab. Theory Related Fields}
\bvolume{125}(\bissue{2}),
\bfpage{271}--\blpage{304}
(\byear{2003})
\doiurl{10.1007/s00440-002-0230-6}
\end{barticle}
\endbibitem

\bibitem[\protect\citeauthoryear{Kov\'{a}cs et~al.}{2015}]{Larsson2015}
\begin{barticle}
\bauthor{\bsnm{Kov\'{a}cs}, \binits{M.}},
\bauthor{\bsnm{Larsson}, \binits{S.}},
\bauthor{\bsnm{Lindgren}, \binits{F.}}:
\batitle{On the backward {E}uler approximation of the stochastic {A}llen-{C}ahn
  equation}.
\bjtitle{J. Appl. Probab.}
\bvolume{52}(\bissue{2}),
\bfpage{323}--\blpage{338}
(\byear{2015})
\doiurl{10.1239/jap/1437658601}
\end{barticle}
\endbibitem

\bibitem[\protect\citeauthoryear{Kov\'{a}cs et~al.}{2018}]{Kovacs2018}
\begin{barticle}
\bauthor{\bsnm{Kov\'{a}cs}, \binits{M.}},
\bauthor{\bsnm{Larsson}, \binits{S.}},
\bauthor{\bsnm{Lindgren}, \binits{F.}}:
\batitle{On the discretisation in time of the stochastic {A}llen-{C}ahn
  equation}.
\bjtitle{Math. Nachr.}
\bvolume{291}(\bissue{5-6}),
\bfpage{966}--\blpage{995}
(\byear{2018})
\doiurl{10.1002/mana.201600283}
\end{barticle}
\endbibitem

\bibitem[\protect\citeauthoryear{Br\'{e}hier et~al.}{2019}]{Brehier2019_2}
\begin{barticle}
\bauthor{\bsnm{Br\'{e}hier}, \binits{C.-E.}},
\bauthor{\bsnm{Cui}, \binits{J.}},
\bauthor{\bsnm{Hong}, \binits{J.}}:
\batitle{Strong convergence rates of semidiscrete splitting approximations for
  the stochastic {A}llen-{C}ahn equation}.
\bjtitle{IMA J. Numer. Anal.}
\bvolume{39}(\bissue{4}),
\bfpage{2096}--\blpage{2134}
(\byear{2019})
\doiurl{10.1093/imanum/dry052}
\end{barticle}
\endbibitem

\bibitem[\protect\citeauthoryear{Br\'{e}hier and
  Gouden\`ege}{2019}]{Brehier2019}
\begin{barticle}
\bauthor{\bsnm{Br\'{e}hier}, \binits{C.-E.}},
\bauthor{\bsnm{Gouden\`ege}, \binits{L.}}:
\batitle{Analysis of some splitting schemes for the stochastic {A}llen-{C}ahn
  equation}.
\bjtitle{Discrete Contin. Dyn. Syst. Ser. B}
\bvolume{24}(\bissue{8}),
\bfpage{4169}--\blpage{4190}
(\byear{2019})
\doiurl{10.3934/dcdsb.2019077}
\end{barticle}
\endbibitem

\bibitem[\protect\citeauthoryear{Qi and Wang}{2019}]{QiRuisheng2019}
\begin{barticle}
\bauthor{\bsnm{Qi}, \binits{R.}},
\bauthor{\bsnm{Wang}, \binits{X.}}:
\batitle{Optimal error estimates of {G}alerkin finite element methods for
  stochastic {A}llen-{C}ahn equation with additive noise}.
\bjtitle{J. Sci. Comput.}
\bvolume{80}(\bissue{2}),
\bfpage{1171}--\blpage{1194}
(\byear{2019})
\doiurl{10.1007/s10915-019-00973-8}
\end{barticle}
\endbibitem

\bibitem[\protect\citeauthoryear{Wang}{2020}]{WangXiaojie2020}
\begin{barticle}
\bauthor{\bsnm{Wang}, \binits{X.}}:
\batitle{An efficient explicit full-discrete scheme for strong approximation of
  stochastic {A}llen-{C}ahn equation}.
\bjtitle{Stochastic Process. Appl.}
\bvolume{130}(\bissue{10}),
\bfpage{6271}--\blpage{6299}
(\byear{2020})
\doiurl{10.1016/j.spa.2020.05.011}
\end{barticle}
\endbibitem

\bibitem[\protect\citeauthoryear{Liu and Qiao}{2021}]{LiuZhihui2021}
\begin{barticle}
\bauthor{\bsnm{Liu}, \binits{Z.}},
\bauthor{\bsnm{Qiao}, \binits{Z.}}:
\batitle{Strong approximation of monotone stochastic partial differential
  equations driven by multiplicative noise}.
\bjtitle{Stoch. Partial Differ. Equ. Anal. Comput.}
\bvolume{9}(\bissue{3}),
\bfpage{559}--\blpage{602}
(\byear{2021})
\doiurl{10.1007/s40072-020-00179-2}
\end{barticle}
\endbibitem

\bibitem[\protect\citeauthoryear{Huang and Shen}{2023}]{Shenjie2023}
\begin{barticle}
\bauthor{\bsnm{Huang}, \binits{C.}},
\bauthor{\bsnm{Shen}, \binits{J.}}:
\batitle{Stability and convergence analysis of a fully discrete semi-implicit
  scheme for stochastic {A}llen-{C}ahn equations with multiplicative noise}.
\bjtitle{Math. Comp.}
\bvolume{92}(\bissue{344}),
\bfpage{2685}--\blpage{2713}
(\byear{2023})
\doiurl{10.1090/mcom/3846}
\end{barticle}
\endbibitem

\bibitem[\protect\citeauthoryear{Yang et~al.}{2024}]{ZhaoWeidong2024}
\begin{barticle}
\bauthor{\bsnm{Yang}, \binits{X.}},
\bauthor{\bsnm{Zhao}, \binits{W.}},
\bauthor{\bsnm{Zhao}, \binits{W.}}:
\batitle{Optimal error estimates of a discontinuous {G}alerkin method for
  stochastic {A}llen-{C}ahn equation driven by multiplicative noise}.
\bjtitle{Commun. Comput. Phys.}
\bvolume{36}(\bissue{1}),
\bfpage{133}--\blpage{159}
(\byear{2024})
\end{barticle}
\endbibitem

\bibitem[\protect\citeauthoryear{Kloeden and Platen}{1992}]{Kloeden1992}
\begin{bbook}
\bauthor{\bsnm{Kloeden}, \binits{P.E.}},
\bauthor{\bsnm{Platen}, \binits{E.}}:
\bbtitle{Numerical Solution of Stochastic Differential Equations}.
\bsertitle{Applications of Mathematics (New York)},
vol. \bseriesno{23},
p. \bfpage{632}
(\byear{1992}).
\doiurl{10.1007/978-3-662-12616-5} .
\burl{https://doi.org/10.1007/978-3-662-12616-5}
\end{bbook}
\endbibitem

\bibitem[\protect\citeauthoryear{Milstein and Tretyakov}{2004}]{Milstein2004}
\begin{bbook}
\bauthor{\bsnm{Milstein}, \binits{G.N.}},
\bauthor{\bsnm{Tretyakov}, \binits{M.V.}}:
\bbtitle{Stochastic Numerics for Mathematical Physics}.
\bsertitle{Scientific Computation},
p. \bfpage{594}
(\byear{2004}).
\doiurl{10.1007/978-3-662-10063-9} .
\burl{https://doi.org/10.1007/978-3-662-10063-9}
\end{bbook}
\endbibitem

\bibitem[\protect\citeauthoryear{Br\'ehier et~al.}{2018}]{Hairer2018}
\begin{barticle}
\bauthor{\bsnm{Br\'ehier}, \binits{C.-E.}},
\bauthor{\bsnm{Hairer}, \binits{M.}},
\bauthor{\bsnm{Stuart}, \binits{A.M.}}:
\batitle{Weak error estimates for trajectories of {SPDE}s under spectral
  {G}alerkin discretization}.
\bjtitle{J. Comput. Math.}
\bvolume{36}(\bissue{2}),
\bfpage{159}--\blpage{182}
(\byear{2018})
\doiurl{10.4208/jcm.1607-m2016-0539}
\end{barticle}
\endbibitem

\bibitem[\protect\citeauthoryear{Br\'{e}hier and Debussche}{2018}]{Brehier2018}
\begin{barticle}
\bauthor{\bsnm{Br\'{e}hier}, \binits{C.-E.}},
\bauthor{\bsnm{Debussche}, \binits{A.}}:
\batitle{Kolmogorov equations and weak order analysis for {SPDE}s with
  nonlinear diffusion coefficient}.
\bjtitle{J. Math. Pures Appl. (9)}
\bvolume{119},
\bfpage{193}--\blpage{254}
(\byear{2018})
\doiurl{10.1016/j.matpur.2018.08.010}
\end{barticle}
\endbibitem

\bibitem[\protect\citeauthoryear{Jacobe~de Naurois et~al.}{2021}]{Naurois2021}
\begin{barticle}
\bauthor{\bsnm{Naurois}, \binits{L.}},
\bauthor{\bsnm{Jentzen}, \binits{A.}},
\bauthor{\bsnm{Welti}, \binits{T.}}:
\batitle{Weak convergence rates for spatial spectral {G}alerkin approximations
  of semilinear stochastic wave equations with multiplicative noise}.
\bjtitle{Appl. Math. Optim.}
\bvolume{84},
\bfpage{1187}--\blpage{1217}
(\byear{2021})
\doiurl{10.1007/s00245-020-09744-6}
\end{barticle}
\endbibitem

\bibitem[\protect\citeauthoryear{Wang and Gan}{2013}]{Wangxiaojie2013}
\begin{barticle}
\bauthor{\bsnm{Wang}, \binits{X.}},
\bauthor{\bsnm{Gan}, \binits{S.}}:
\batitle{Weak convergence analysis of the linear implicit {E}uler method for
  semilinear stochastic partial differential equations with additive noise}.
\bjtitle{J. Math. Anal. Appl.}
\bvolume{398}(\bissue{1}),
\bfpage{151}--\blpage{169}
(\byear{2013})
\doiurl{10.1016/j.jmaa.2012.08.038}
\end{barticle}
\endbibitem

\bibitem[\protect\citeauthoryear{Conus et~al.}{2019}]{Jentzen2019}
\begin{barticle}
\bauthor{\bsnm{Conus}, \binits{D.}},
\bauthor{\bsnm{Jentzen}, \binits{A.}},
\bauthor{\bsnm{Kurniawan}, \binits{R.}}:
\batitle{Weak convergence rates of spectral {G}alerkin approximations for
  {SPDE}s with nonlinear diffusion coefficients}.
\bjtitle{Ann. Appl. Probab.}
\bvolume{29}(\bissue{2}),
\bfpage{653}--\blpage{716}
(\byear{2019})
\doiurl{10.1214/17-AAP1352}
\end{barticle}
\endbibitem

\bibitem[\protect\citeauthoryear{Jentzen and Kurniawan}{2021}]{Jentzen2021}
\begin{barticle}
\bauthor{\bsnm{Jentzen}, \binits{A.}},
\bauthor{\bsnm{Kurniawan}, \binits{R.}}:
\batitle{Weak convergence rates for {E}uler-type approximations of semilinear
  stochastic evolution equations with nonlinear diffusion coefficients}.
\bjtitle{Found. Comput. Math.}
\bvolume{21}(\bissue{2}),
\bfpage{445}--\blpage{536}
(\byear{2021})
\doiurl{10.1007/s10208-020-09448-x}
\end{barticle}
\endbibitem

\bibitem[\protect\citeauthoryear{Andersson et~al.}{2016}]{Andersson2016_2}
\begin{barticle}
\bauthor{\bsnm{Andersson}, \binits{A.}},
\bauthor{\bsnm{Kruse}, \binits{R.}},
\bauthor{\bsnm{Larsson}, \binits{S.}}:
\batitle{Duality in refined {S}obolev-{M}alliavin spaces and weak approximation
  of {SPDE}}.
\bjtitle{Stoch. Partial Differ. Equ. Anal. Comput.}
\bvolume{4}(\bissue{1}),
\bfpage{113}--\blpage{149}
(\byear{2016})
\doiurl{10.1007/s40072-015-0065-7}
\end{barticle}
\endbibitem

\bibitem[\protect\citeauthoryear{Wang}{2016}]{WangXiaojie2016}
\begin{barticle}
\bauthor{\bsnm{Wang}, \binits{X.}}:
\batitle{Weak error estimates of the exponential {E}uler scheme for semi-linear
  {SPDE}s without {M}alliavin calculus}.
\bjtitle{Discrete Contin. Dyn. Syst.}
\bvolume{36}(\bissue{1}),
\bfpage{481}--\blpage{497}
(\byear{2016})
\doiurl{10.3934/dcds.2016.36.481}
\end{barticle}
\endbibitem

\bibitem[\protect\citeauthoryear{Cui and Hong}{2019}]{CuiJianbo2019}
\begin{barticle}
\bauthor{\bsnm{Cui}, \binits{J.}},
\bauthor{\bsnm{Hong}, \binits{J.}}:
\batitle{Strong and weak convergence rates of a spatial approximation for
  stochastic partial differential equation with one-sided {L}ipschitz
  coefficient}.
\bjtitle{SIAM J. Numer. Anal.}
\bvolume{57}(\bissue{4}),
\bfpage{1815}--\blpage{1841}
(\byear{2019})
\doiurl{10.1137/18M1215554}
\end{barticle}
\endbibitem

\bibitem[\protect\citeauthoryear{Cui et~al.}{2021}]{CuiJianbo2021}
\begin{barticle}
\bauthor{\bsnm{Cui}, \binits{J.}},
\bauthor{\bsnm{Hong}, \binits{J.}},
\bauthor{\bsnm{Sun}, \binits{L.}}:
\batitle{Weak convergence and invariant measure of a full discretization for
  parabolic {SPDE}s with non-globally {L}ipschitz coefficients}.
\bjtitle{Stochastic Process. Appl.}
\bvolume{134},
\bfpage{55}--\blpage{93}
(\byear{2021})
\doiurl{10.1016/j.spa.2020.12.003}
\end{barticle}
\endbibitem

\bibitem[\protect\citeauthoryear{Cai et~al.}{2021}]{WangXiaojie2021}
\begin{barticle}
\bauthor{\bsnm{Cai}, \binits{M.}},
\bauthor{\bsnm{Gan}, \binits{S.}},
\bauthor{\bsnm{Wang}, \binits{X.}}:
\batitle{Weak convergence rates for an explicit full-discretization of
  stochastic {A}llen-{C}ahn equation with additive noise}.
\bjtitle{J. Sci. Comput.}
\bvolume{86}(\bissue{3}),
\bfpage{34}--\blpage{30}
(\byear{2021})
\doiurl{10.1007/s10915-020-01378-8}
\end{barticle}
\endbibitem

\bibitem[\protect\citeauthoryear{Da~Prato et~al.}{2019}]{DaPrato2019}
\begin{barticle}
\bauthor{\bsnm{Da~Prato}, \binits{G.}},
\bauthor{\bsnm{Jentzen}, \binits{A.}},
\bauthor{\bsnm{R\"{o}ckner}, \binits{M.}}:
\batitle{A mild {I}t\^{o} formula for {SPDE}s}.
\bjtitle{Trans. Amer. Math. Soc.}
\bvolume{372}(\bissue{6}),
\bfpage{3755}--\blpage{3807}
(\byear{2019})
\doiurl{10.1090/tran/7165}
\end{barticle}
\endbibitem

\bibitem[\protect\citeauthoryear{Da~Prato and Zabczyk}{2014}]{DaPrato2014}
\begin{bbook}
\bauthor{\bsnm{Da~Prato}, \binits{G.}},
\bauthor{\bsnm{Zabczyk}, \binits{J.}}:
\bbtitle{Stochastic Equations in Infinite Dimensions},
\bedition{2}nd edn.
\bsertitle{Encyclopedia of Mathematics and its Applications},
vol. \bseriesno{152},
p. \bfpage{493}
(\byear{2014}).
\doiurl{10.1017/CBO9781107295513} .
\burl{https://doi.org/10.1017/CBO9781107295513}
\end{bbook}
\endbibitem

\bibitem[\protect\citeauthoryear{Thom\'{e}e}{2006}]{Thomee2006}
\begin{bbook}
\bauthor{\bsnm{Thom\'{e}e}, \binits{V.}}:
\bbtitle{Galerkin Finite Element Methods for Parabolic Problems},
\bedition{2}nd edn.
\bsertitle{Springer Series in Computational Mathematics},
vol. \bseriesno{25},
p. \bfpage{370}
(\byear{2006})
\end{bbook}
\endbibitem

\bibitem[\protect\citeauthoryear{Gawarecki and Mandrekar}{2011}]{Gawarecki2011}
\begin{bbook}
\bauthor{\bsnm{Gawarecki}, \binits{L.}},
\bauthor{\bsnm{Mandrekar}, \binits{V.}}:
\bbtitle{Stochastic Differential Equations in Infinite Dimensions with
  Applications to Stochastic Partial Differential Equations}.
\bsertitle{Probability and its Applications (New York)},
p. \bfpage{291}
(\byear{2011}).
\doiurl{10.1007/978-3-642-16194-0} .
\burl{https://doi.org/10.1007/978-3-642-16194-0}
\end{bbook}
\endbibitem

\bibitem[\protect\citeauthoryear{Le\'{o}n and Nualart}{1998}]{Nualart1998}
\begin{barticle}
\bauthor{\bsnm{Le\'{o}n}, \binits{J.A.}},
\bauthor{\bsnm{Nualart}, \binits{D.}}:
\batitle{Stochastic evolution equations with random generators}.
\bjtitle{Ann. Probab.}
\bvolume{26}(\bissue{1}),
\bfpage{149}--\blpage{186}
(\byear{1998})
\doiurl{10.1214/aop/1022855415}
\end{barticle}
\endbibitem

\bibitem[\protect\citeauthoryear{Higham et~al.}{2002}]{Higham2002}
\begin{barticle}
\bauthor{\bsnm{Higham}, \binits{D.J.}},
\bauthor{\bsnm{Mao}, \binits{X.}},
\bauthor{\bsnm{Stuart}, \binits{A.M.}}:
\batitle{Strong convergence of {E}uler-type methods for nonlinear stochastic
  differential equations}.
\bjtitle{SIAM J. Numer. Anal.}
\bvolume{40}(\bissue{3}),
\bfpage{1041}--\blpage{1063}
(\byear{2002})
\doiurl{10.1137/S0036142901389530}
\end{barticle}
\endbibitem

\bibitem[\protect\citeauthoryear{Br\'{e}hier and
  Gouden\`ege}{2020}]{Brehier2020}
\begin{barticle}
\bauthor{\bsnm{Br\'{e}hier}, \binits{C.-E.}},
\bauthor{\bsnm{Gouden\`ege}, \binits{L.}}:
\batitle{Weak convergence rates of splitting schemes for the stochastic
  {A}llen-{C}ahn equation}.
\bjtitle{BIT}
\bvolume{60}(\bissue{3}),
\bfpage{543}--\blpage{582}
(\byear{2020})
\doiurl{10.1007/s10543-019-00788-x}
\end{barticle}
\endbibitem

\end{thebibliography}
\end{document}